\documentclass[11pt]{article}
\usepackage{geometry}                
\usepackage{graphicx}
\usepackage{diagrams}
\usepackage{amsmath,mathtools,listings,amsfonts,amsthm,stmaryrd,rotating}
\usepackage{amssymb}
\usepackage{epstopdf}
\def\TG{\mathcal{TG}}
\newtheorem{definition}{Definition}
\newtheorem{proposition}{Proposition}
\newtheorem{corollary}{Corollary}
\newtheorem{example}{Example}
\usepackage{xcolor}
\newcommand*{\tran}{^{\mkern-1.5mu\mathsf{T}}}
\newtheorem{lemma}{Lemma}
\newtheorem{theorem}{Theorem}
\def\tr{\mathop{\mathrm{tr}}}
 \def\R{\mathbb{R}}
 \def\T{\mathbb{T}}
 
\title{Principal Symmetric Space Analysis}
\author{Stephen R Marsland, Robert I McLachlan, and Charles Curry}

\begin{document}
\maketitle
\section{Introduction}
Principal Components Analysis (PCA \cite{jolliffe}), traditionally applied for data on a Euclidean space $E^n$, has many notable features that have made it one of the most widely used of all statistical techniques. We single out the following:
\begin{enumerate}
\item The approximating subspaces (affine subspaces of $E^n$) have zero extrinsic curvature;
\item any two affine subspaces of the same dimension are related by a Euclidean transformation;
\item the best approximations of each dimension are nested (that is, the best approximation by a $k$-dimensional subspace lies in the best approximation by a $k+1$-dimensional subspace); and
\item the best approximations of each dimension from 0 to $n-1$ can be computed easily using linear algebra.
\end{enumerate}
The underlying idea of PCA has been extended to deal with data on non-Euclidean manifolds.
One such method is that of Principal Geodesic Analysis (PGA \cite{PGA,DiffTens,gebhart}). For data on a Riemannian manifold
$M$,  the Karcher mean $x$ is computed and the data pulled back to the tangent
space $T_xM$ by the logarithm of the Riemannian exponential map at $x$ (see Figure \ref{fig:pga}). PCA can
now be applied to the data on this Euclidean vector space. However, this and related methods
suffer from a fundamental flaw in that they fail to deal properly with the curvature of the manifold.
Two geodesics with  common base points and distant tangent vectors may pass close to each other
or intersect (see Figure \ref{fig:pga}). In this situation, nearby data points would become far apart in their linear approximation.

\begin{figure}
\begin{center}
\includegraphics[width=5cm]{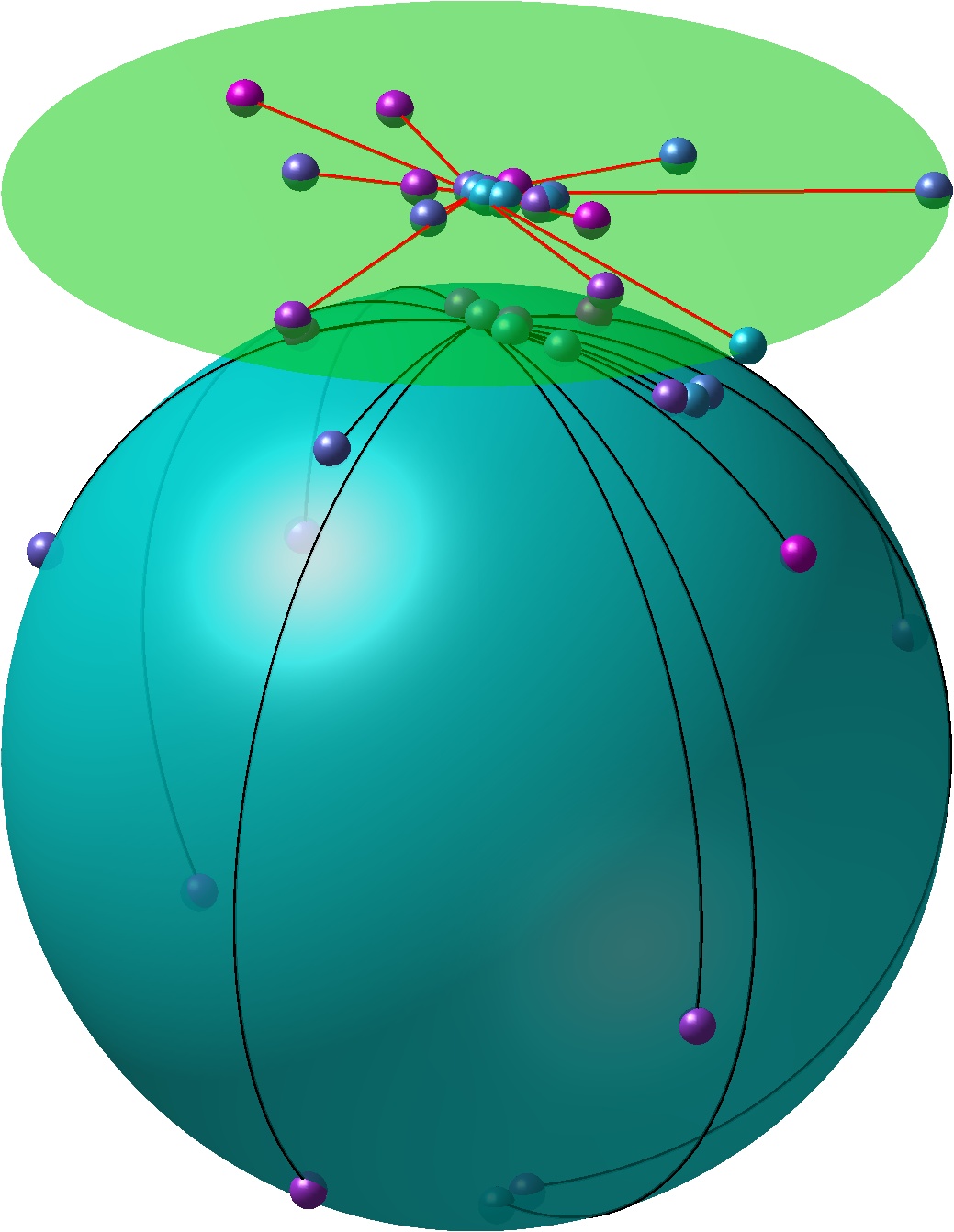}
\end{center}
\caption{\label{fig:pga} In Principal Geodesic Analysis, the data (here 20 points on a sphere) is pulled back to the
tangent space of the Karcher mean (shown here as a disk) using geodesics. Data points near the mean are
well represented, but data points far from the mean (here, near the south pole) become far apart in the linear approximation.}
\end{figure}

In seeking a method that avoids this flaw we have focussed on property (1) above. A submanifold $N$ of a Riemannian manifold $M$ has zero extrinsic curvature 
if and only
if it is  totally geodesic (i.e., any geodesic of $N$ is also a geodesic of $M$). Such submanifolds provide
excellent approximating spaces, being in a sense the flattest or simplest possible lower-dimensional
representations of the data. One-dimensional totally geodesic submanifolds are geodesics, which 
are widely used for 1-dimensional interpolation and data fitting on manifolds \cite{fletcher,kenobi,rent}. 

Generic manifolds have no totally geodesic submanifolds of dimension higher than 1, but Riemannian symmetric
spaces  have many. Examples
of Riemannian symmetric spaces are compact Lie groups, Euclidean spaces, spheres, projective spaces,  Grassmannians, and products of these; these are all important examples of nonlinear domains for data.  We will see
that the structure of totally geodesic submanifolds offers rich possibilities for data reduction
and for the discovery of hidden structure in data sets. Totally geodesic submanifolds of Riemannian symmetric spaces are themselves Riemannian symmetric, which offers the possibility of a nested structure as in point (3) above. 

Although some form of nesting is desirable, we will see that, given two best approximating
totally geodesic submanifolds, one is not necessarily contained in the other.
To overcome this, in 
this paper we define the {\em symmetric space approximations} of a dataset in a Riemannian symmetric space. This is a set whose elements are best approximating totally geodesic submanifolds. 
Applying this construction recursively gives the {\em principal symmetric space approximation} which is structured as a rooted tree. In this sense the nesting structure is retained, although
it may be complicated in specific instances.

In Section \ref{sec:symspace} we review the relevant elements of symmetric spaces. In particular,
the determination of totally geodesic submanifolds can be reduced to a purely algebraic equation in a vector space (the Lie algebra of the symmetry group of the symmetric space). Solving this equation may be difficult, however; it has been solved completely only for spaces of rank 1 (such as spheres and projective spaces) and rank 2 (such as 2-Grassmannians and products of two spheres). 
In the remainder of the paper, therefore, we proceed by example. Section \ref{sec:sphere} considers data on the $n$-dimensional sphere $S^n$. Section \ref{sec:grass} considers
data on the Grassmannian $G(k,n)$ of $k$-planes in $\mathbb{R}^n$; even here we need to restrict to the simple submanifolds $G(k,m)$ of $k$-planes in $\mathbb{R}^m$. In both of these cases,
we will show that the approximation problem can be linearised so that a PCA-like nested
sequences of approximating submanifolds can be determined using linear algebra.

More complicated cases are handled in Section \ref{sec:products} on products of spheres.
The two subcases that we consider are tori $(S^1)^n$ and polyspheres $(S^2)^n$. 
Each of these has an infinite number of distinct types of totally geodesic submanifolds and
each reveals new features of the general situation. 

We now introduce the central ideas of principal symmetric space approximation.

Let $M=G/H$ be a Riemannian symmetric space. Let $\TG(M)$ be the set of connected totally
geodesic submanifolds of $M$. $G$ acts on $\TG(M)$ and partitions it into group orbits. We regard the submanifolds
in each orbit as being of equivalent structure and complexity, so that if there is a unique best
approximation within an orbit, we choose it; but the submanifolds from different orbits,
even if of the same dimension, are different geometrically and are best regarded as representing
different models. 

Let the data set be $X:=(x_1,\dots,x_d)$, where $x_i\in M$, and 
let $N\in\TG(M)$ be a totally geodesic submanifold of $M$. Let 
\[ d(x,N) = \min_{y\in N} d(x,y)\]
and
\[ d(X,N)^2 =\sum_{i=1}^d d(x_i,N)^2\]
\begin{definition}
The {\em symmetric space approximations}  of $X\in M^d$ {\em with respect to $M$} are the elements of
\[
\mathcal{SSA}(X,M) := \left\{ gN\colon g = \mathop{\rm argmin}_{g\in G} d(X,gN),\quad N\in\TG(M)\right\}
\]
where local minima are taken.
\end{definition}
Thus, each  element of $\mathcal{SSA}(X,M)$ is a totally geodesic submanifold $N$ of $M$, which best approximates the data in the sense that the approximation cannot be improved by passing to $gN$ where $g$ is close to the identity.

As each $N\in\mathcal{SSA}(X,M)$ is a Riemannian symmetric space, it typically has many totally geodesic
submanifolds itself. These are already contained in $\TG(M)$. We can now calculate the symmetric
space approximations of $X$ with respect to each such $N$. Repeating this construction gives a tree of
submanifolds. Each branch contains a nested sequence of approximations of decreasing dimensions, with
each branch terminating in a submanifold of dimension 0, that is, a point.

\begin{definition}
The {\em principal symmetric space approximation} $\mathcal{PSSA}(X,M)$ of $X$ with respect to $M$
is the rooted tree in which
\begin{enumerate}
\item each node is a totally geodesic submanifold of $M$;
\item the root node is $M$; and
\item the children of a node $N$ are the symmetric space approximations of $X$ with respect to $N$.
\end{enumerate}
\end{definition}

Examples are the unbranched tree $E^n\supset E^{n-1}\supset \dots\supset E^0$ found in Euclidean PCA,
and the 2-node tree $M\supset\{x\}$ for any Riemannian manifold $M$, where $x$ is the Karcher mean of $X$.


\section{Symmetric spaces}
\label{sec:symspace}
We give a brief account of symmetric spaces relevant to the sequel. The material presented here is standard, see for instance \cite[Chapter XI]{KN}.
\begin{definition}
A symmetric space is a triple $(G,H,\sigma)$ where $G$ is a connected Lie group, $\sigma$ is an involutive automorphism of $G$, and $H$ is  a closed subgroup of $G$ such that $H$ lies between the isotropy subgroup $G_{\sigma}$ and its identity component $G_{\sigma}^{o}$. 
\end{definition}

In particular, the manifold $M=G/H$ is a canonically reductive homogeneous space and hence comes equipped with a canonical linear connection. Let $s_o$ be the automorphism of $G/H$ induced by $\sigma$. For any point $x=g.o$ where $o$ is the origin, the mapping $s_x = g.s_o.g^{-1}$ is independent of the choice of $g$. Moreover, $s_x$ is a symmetry of the canonical connection for all $x$, i.e. a diffeomorphism of a neighbourhood of $x$ onto itself sending $\exp{X}\rightarrow\exp{-X}$ for any tangent vector $X$. We now present the infinitesimal picture.

\begin{definition}
A symmetric Lie algebra is a triple $(\mathfrak{g},\mathfrak{h},\sigma)$ where $\mathfrak{g}$ is a Lie algebra, $\sigma$ is an involutive automorphism of $\mathfrak{g}$, and  $\mathfrak{h}\subset\mathfrak{g}$ is the Lie subalgebra of elements fixed by $\sigma$.
\end{definition}

There is a one-to-one correspondence between effective symmetric Lie algebras and almost effective (i.e., the only normal subgroups of $G$ are discrete) symmetric spaces with $G$ simply connected and $H$ connected.

The involution $\sigma$ induces a decomposition $\mathfrak{g}=\mathfrak{h}+\mathfrak{m}$ of $\mathfrak{g}$ into the $\pm 1$ eigenspaces of $\sigma$, called the canonical decomposition.  The following relations the hold, which suffice to characterize symmetric Lie algebras:
\[
[\mathfrak{h},\mathfrak{h}] \subset \mathfrak{h},\quad
[\mathfrak{h},\mathfrak{m}] \subset \mathfrak{m},\quad
[\mathfrak{m},\mathfrak{m}] \subset \mathfrak{h}.
\]
Examples of symmetric spaces include the oriented Grassmannian $\mathrm{G}_+(k,n)$ of oriented $k$-planes in $\mathbb{R}^n$. The symmetric space structure is described by $\mathrm{G}_+(k,n) \simeq SO(n)/(SO(k)\times SO(n-k))$, with automorphism $\sigma(A)=SAS^{-1}$,
\[
S=\begin{pmatrix}
-I_p & 0 \\
0 & I_q
\end{pmatrix},
\]  
where $I_p$ is the $p\times p$ identity matrix. The case $k=1$ gives the symmetric space structure of the sphere $S^n$. The unoriented case $\mathrm{G}(k,n)\simeq O(n)/(O(k)\times O(n-k))$ is similar, and specialization to $k=1$ then gives projective spaces. We also note that there is a natural direct product of symmetric spaces: $(G,H,\sigma)\times (G',H',\sigma')=(G\times G',H\times H',\sigma\times\sigma')$.

A submanifold $N\subset M$ is said to be totally geodesic if for all points $x\in N$ and tangent vectors $X\in T_x(N)$, the geodesic $\exp(tX)$ is contained in $N$ for sufficiently small $t$. Where $M$ is a Riemannian manifold, this is equivalent to requiring that the induced metric on $N$ coincides with the metric on $M$. A Lie triple system $\mathfrak{m}$ is a subspace of a Lie algebra for which $[[\mathfrak{m},\mathfrak{m}],\mathfrak{m}]\subset \mathfrak{m}$. The following result underlies our interest in totally geodesic submanifolds of symmetric spaces:
\begin{theorem}
Let $(G,H,\sigma)$ be a symmetric space with symmetric Lie algebra $\mathfrak{g}=\mathfrak{h}+\mathfrak{m}$. There is a one-to-one correspondence between complete totally geodesic submanifolds $M'$ containing the origin and Lie triple systems $\mathfrak{m'}\subset \mathfrak{m}$. Moreover $(G',H',\sigma')$ is a symmetric subspace, where $G'$ is the largest connected Lie subgroup of $G$ leaving $M'$ invariant, $H'=G'\cap H$, and $\sigma'=\sigma|_{G'}$.
\end{theorem}
Note that the proof constructs the symmetric subalgebra $(\mathfrak{g}',\mathfrak{h}',\sigma')$. Indeed, given such an $\mathfrak{m}'$, we take $\mathfrak{h}'=[\mathfrak{m}',\mathfrak{m}']$, then set $\mathfrak{g}'=\mathfrak{h}'+\mathfrak{m'}$. 

The problem of classifying totally geodesic submanifolds of symmetric spaces is thus reduced to an algebraic one. It remains a difficult task \cite{CN,Klein,wolf}. Moreover, there may exist complicated totally geodesic submanifolds which are of little physical relevance, so in some cases we restrict our attention to subfamilies of symmetric subspaces.

The notion of a symmetric space approximation requires a distance function on the manifold. It is most natural to specify this through a Riemannian metric. This makes most sense where our notion of totally geodesic submanifold coincides with the Riemannian geodesics, as summarized by the following definition.  

\begin{definition}
A Riemannian symmetric space is a symmetric space for which the canonical connection coincides with the Riemannian (Levi-Civita) connection.
\end{definition}
This implies that the symmetries $s_x$ are isometries. A symmetric space equipped with a metric is Riemannian symmetric if the metric is $G$-invariant. Given a symmetric space $(G,H,\sigma)$ for which $\mathrm{ad}_{\mathfrak{g}}(H)$ is compact, a $G$-invariant Riemannian metric may be constructed in a canonical manner.

All of the symmetric spaces we consider are canonically Riemannian symmetric spaces. Nonetheless, for practical purposes we will minimize distances which differ from the Riemannian distance, typically to obtain a linearization of the minimization problem. We will say that two metrics $d_1$, $d_2$ are {\em compatible} if they agree
up to first order for nearby points, that is, if $d_2(x,y) = d_1(x,y) + \mathcal{O}(d_1(x,y)^2)$. 
In a non-Riemannian metric space, the length of a curve is defined by a Riemann sum, and thus
one still has the concept of geodesic and of totally geodesic submanifolds in this case. 
Moreover, the geodesics and totally geodesic submanifolds of a given smooth manifold equipped
with two compatible metrics coincide. Thus, although perturbing the Riemannian metric of a Riemannian 
symmetric space changes the {\em specific} principal symmetric space approximation corresponding
to a given set of data, it does not change the system of totally geodesic submanifolds itself.

\section{Spheres}
\label{sec:sphere}
Datasets on high-dimensional spheres arise naturally whenever we have a set of measurements in a Euclidean space for which the magnitude is irrelevant. One important instance concerns directional data, see \cite{PNS} and the references therein for more examples.

The connected totally geodesic submanifolds of $S^n$ are the spheres $S^k$, realized as the image
of a standard sphere $x_1^2+\dots+x_{k+1}^2=1$ in $\R^{n+1}$ under an element of $SO(n+1)$
\cite[Thm 1]{wolf}.

We consider first the case of $S^2$, represented as the set of unit vectors in $\mathbb{R}^3$. 
Geodesics on $S^2$ are precisely the great circles, which may be described as the set of points in $S^2$ orthogonal to a given unit vector $v$. We call this great circle $S_v := \{w\in\mathbb{R}^3\colon w\cdot v=0\}$.
In this case the Riemannian distance from a point $x$ to $S_v$ is
the angle between $x$ and $v$, that is,
\[
d(x,S_{{v}}) = \sin^{-1}(\left|x\cdot {v}\right|).
\]
Note that the great circle with axis $v$ consists of the intersection of $S^2$ and the plane with normal vector ${v}$. More generally, the totally geodesic submanifolds of $S^n$, viewed as submanifolds of $\mathbb{R}^{n+1}$, are precisely the intersections $S_N$ of $S^n$ with a given linear subspace $N$ of $\mathbb{R}^{n+1}$. 

\begin{lemma}
\label{lem:s1}
Let $N$ be a subspace of $\mathbb{R}^{n+1}$ and let  $\{{v}_1,\ldots,{v}_m\}$ be an orthogonal basis for $N^\perp$. Then the distance between $x\in S^n$ and $S_N:=S^n\cap N$ is
\[
d(x,S_N) = \sin^{-1}\left(\sqrt{\sum_i (x\cdot {v}_i)^2}\right).
\]
\end{lemma}
\begin{proof}
The angle $\theta$ ($=d(x,S_N)$) between $x$ and $N$, and the angle $\hat\theta$ between $x$ and $N^\perp$, are complementary angles.
Likewise, the angle $\theta$ between $x$ and $S_N$ and the angle $\hat\theta$ between $x$ and $S^n\cap N^\perp$ are complementary angles. 
Let $\hat x$ be the orthogonal projection of $x$ to $N^\perp$, that is, $\hat x = \sum_{i=1}^m (x\cdot v_i) v_i$.
Then 
$$\sin\theta = \cos\hat\theta = \|\hat x\| = \sqrt{\sum_{i=1}^m(x\cdot v_i)^2}.$$
\end{proof}

We now make the obvious linearization of this distance so that best approximations may be determined using linear algebra. 
We call the {\em projection distance} $d_p(x,v)$ between $x$ and $v\in S^n$ the shortest Euclidean distance from $x$
to $\mathrm{span}(v)$ in $\mathbb{R}^{n+1}$. (Equivalently, from $v$ to $\mathrm{span}(x)$.) 

\begin{lemma}
Let $N$ be a subspace of $\mathbb{R}^{n+1}$ and let  $\{{v}_1,\ldots,{v}_m\}$ be an orthogonal basis for $N^\perp$. Then the projection distance of between $x\in S^n$ and $S_N$ is
\[
d_p(x,S_N) = \sqrt{\sum_{i=1}^m(x\cdot v_i)^2}.
\]
\end{lemma}
\begin{proof}
Let $\theta$ be the angle between $x$ and $v\in S^n$, that is, $\cos\theta = x\cdot v$. 
Then $d_p(x,v)=\sin\theta$. The same construction as in Lemma \ref{lem:s1}, except measuring
distances as $\sin\theta$ instead of $\theta$, gives the result.
\end{proof}
Note that the projection distance between two points is nonlinear; its use is favoured here because
it becomes linear when calculating distances to subspheres $S_N$.
The projection distance is compatible with the Riemannian distance.

\begin{proposition}
\label{prop:s}
Let $X$ be the matrix whose columns consists of the data points $x_i\in S^n$, where $S^n$ is identified with the unit sphere in $\mathbb{R}^{n+1}$. Then for any $m$ with $0<m<n$, the best approximating $(n-m)$-sphere in 
the projection distance is given by $S_N$, where $N$ is the the span of the singular vectors corresponding to the smallest $m$ singular values of $X\tran $.
\end{proposition}

\begin{proof}
Let $V=[v_1,\dots,v_m]\in\mathbb{R}^{n\times m}$. 
We have $d_p(x_j,S_N)^2=\sum_{i=1}^m (x_j\cdot v_i)^2$, and
thus \[
d_p(X,S_N)^2 = \sum_{i=1}^m\sum_{j=1}^d (x_j\cdot v_i)^2 = \|X\tran  V\|_F^2.\]
We seek to minimize $d_p(X,S_N)$ subject
to the constraint that $V$ is orthogonal. Introducing a Lagrange multiplier $\Lambda\in\mathbb{R}^{m\times m}$
for the constraint, where $\Lambda\tran =\Lambda$, we need to make
\[
\| X\tran  V \|_F^2 + \tr\Lambda(V\tran  V - I) = \tr(V\tran  X X\tran  V) + \tr\Lambda(V\tran  V- I)
\]
stationary in $W$. The variational equations are
\[
X X\tran  V = V \Lambda,\quad V\tran  V = I.
\]
At any solution to these equations, the objective function is $d_p(X,S_N)^2=\tr(V\tran  X X\tran  V) = \tr(V\tran  V\Lambda) = \tr(\Lambda)$.
Given any solution $(V,\Lambda)$ to these equations, orthogonally diagonalize $\Lambda=Z \Omega Z\tran $
where $Z\tran  Z = I$ and $\Omega$ is diagonal. Then $(WZ, \Omega)$ is also a solution. The value
of the objective function, $\tr \Lambda = \tr\Omega$, is the same for both solutions. Therefore,
we can take $\Lambda$ to be diagonal. Therefore, the stationary points are those for which
the columns of $V$ are eigenvectors of $XX\tran $ (that is, singular vectors of $X\tran $) and the diagonal entries of
$\Lambda$ are the associated eigenvalues of $X X\tran $ (that is, squares of the singular values of $X\tran $).
The minimum value of the objective function is obtained by taking the $m$ smallest singular values.
\end{proof}

Note that, in the sense of Euclidean PCA, if we regard the data as a set of
points in $\mathbb{R}^{n+1}$, the best approximating $(n+1-m)$-subspace is just the
span of the singular vectors associated with the $n+1-m$ {\em largest} singular values
of $X\tran $. That subspace is the orthogonal complement of the span of the singular
vectors associated with the $m$ smallest singular values, found in the proposition.
Thus in this case, the two approximations coincide (after intersecting with $S^n$).

The linearization of the distance function, considered here, reduces the calculation to linear algebra,
produces unique best approximations, and also provides the
nesting property shared by Euclidean PCA: the best $S^p$ lies inside
the best $S^m$ for $p<m$:

\begin{corollary}
\label{cor:s2}
The principal symmetric space approximation of $X$ with respect to $S^n$ is
the unbranched tree $S^n\supset S^{n-1}\supset\dots\supset S^1$, where each $S^m$ is
as determined in Proposition \ref{prop:s}.
\end{corollary}

As stated, we have restricted the dimension of the subspheres in Proposition \ref{prop:s} and Corollary
\ref{cor:s2} to be positive. If they are applied with $m=n$ to yield $0$-dimensional approximations,
they yield the {\em pair} of antipodal points that best approximates the data in the projective metric,
because $S_N$ is then disconnected. Depending on the application, this may be what is wanted.
Even if the best {\em single} point is wanted, the best such $S_N$ may still be a usefully good approximation
if the data is, in fact, strongly clustered around a single point. If the data is not strongly clustered,
and the best single point is wanted, then it may be necessary to switch to another
metric (e.g. the Riemannian metric) and calculate the best point within each
$S^m$ in Cor. \ref{cor:s2}, creating a branched tree of approximations.

\begin{example}\rm
\label{ex:1}
We present a sequence of 3 examples of synthetic datasets on $S^3$. Each contains 20 data points. In the first dataset, each $x_i$ is the projection of
a point in $N(0,\mathrm{diag}(1,1,0.1,0.05))$ to $S^3$. The data lies close
to the 2-sphere $x_4=0$ and even closer to the great circle $x_3=x_4=0$.  
The singular values of $X\tran $ were found to be $(0.3486,0.4095,3.0571,3.2195)$.
Thus the error of the best $S^2$ approximation to the data is 0.3486, and
the error of the best $S^1$ approximation (shown in Fig. \ref{fig:sphere1}, left) is
$0.5378$. 
The error of the best $S^0$ approximation is $3.1040$ and is clearly found
to be not relevant. Likewise, the Karcher mean is not relevant for this dataset.
 
In the second dataset, each $x_i$ is the projection of
a point in $N(0,\mathrm{diag}(1,0.3,0.1,0.05))$ to $S^3$.
Thus the data are more strongly clustered around the 0-sphere $x_1=\pm1,\ x_2=x_3=x_4=0$.
This is revealed in the singular values $(0.3339, 0.9568, 1.2380, 4.1762)$. The
error of the best $S^2$, $S^1$, and $S^0$ approximations are 0.3339, 1.0134, and 1.5998, respectively. The nested approximations are shown in Fig. \ref{fig:sphere1} (centre). 
The Karcher mean is not relevant for this dataset.

In the third dataset, each $x_i$ is the projection of $N([1,0,0,0],\mathrm{diag}([0,0.4,0.1,0.05])$
to $S^3$ and is thus clustered around the point $(1,0,0,0)$. 
The singular values are $(0.1712,0.3685,1.5513,4.1747)$.
The nested approximations are shown in Fig. \ref{fig:sphere1} (right). 
The Karcher mean (which here coincides with the best $S^0$, in the appropriate metric)
is relevant for this dataset.
\end{example}

\begin{figure}
\begin{center}
\includegraphics[width=4.3cm]{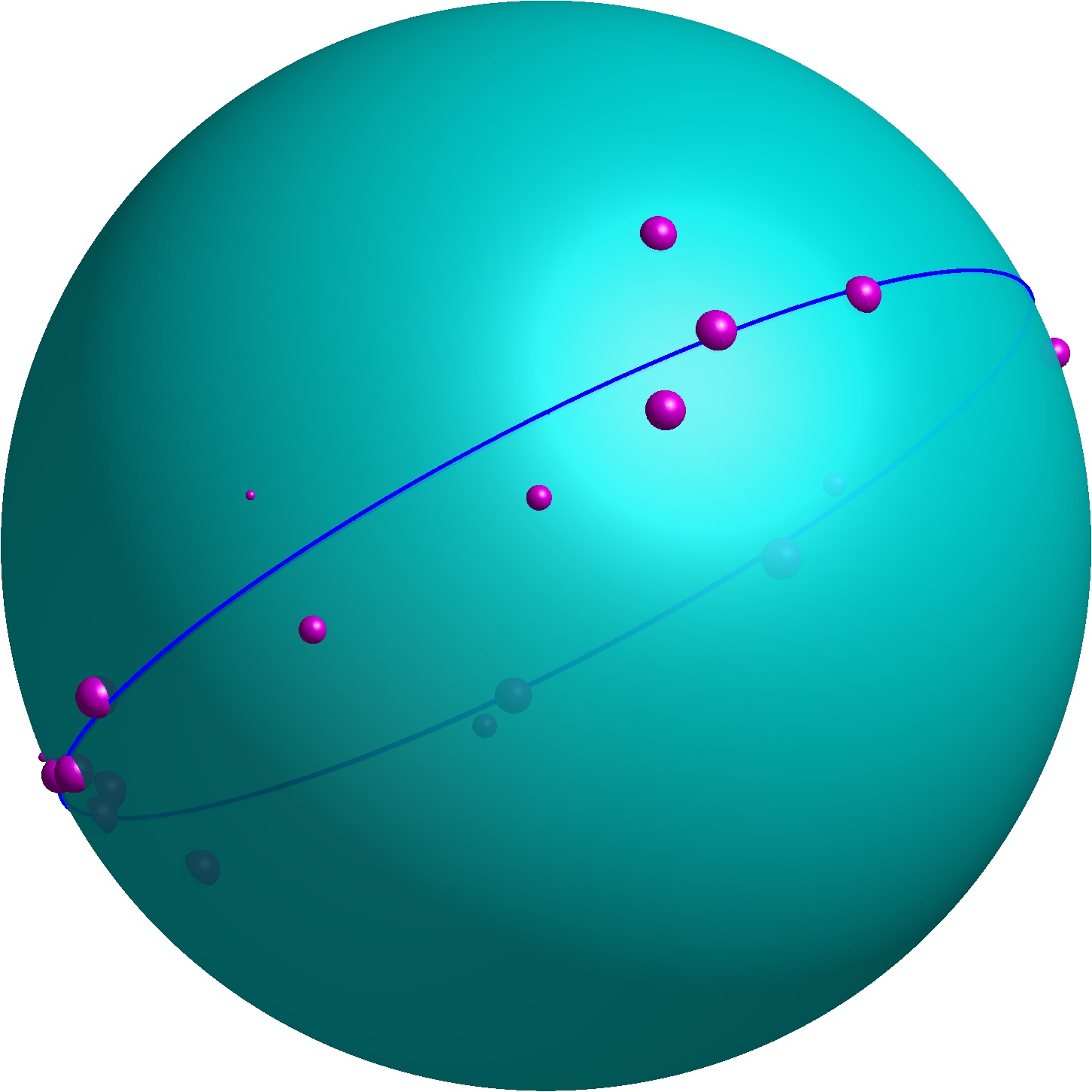}\hspace{5mm}
\includegraphics[width=4.3cm]{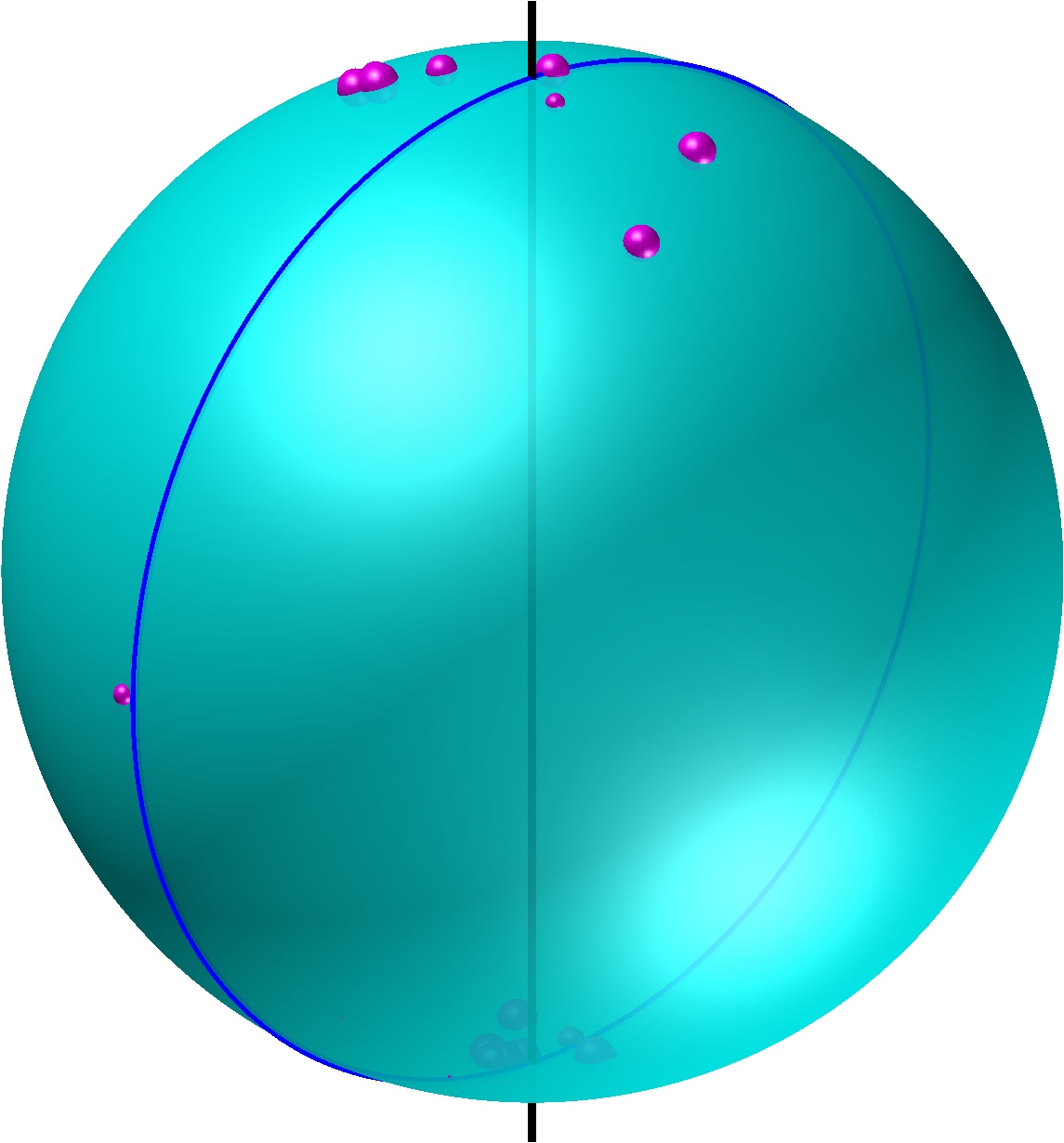}\hspace{5mm}
\includegraphics[width=4.3cm]{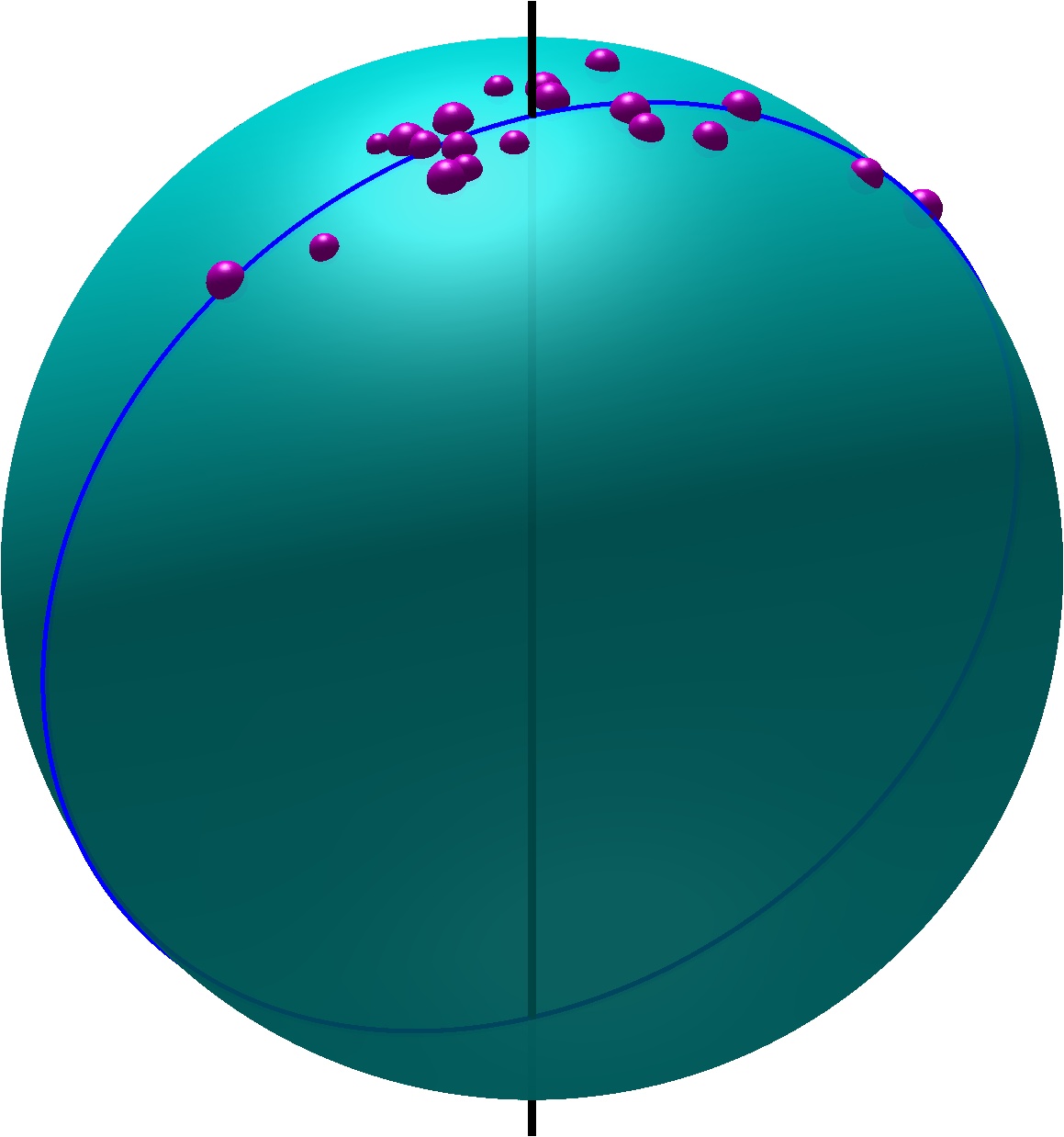}
\end{center}
\caption{\label{fig:sphere1}
Results for datasets 1--3 of Example \ref{ex:1}. In each case, the best subspheres that approximate a set of 20 points on $S^3$ is shown. Data points further from the best $S^2$ are shown smaller. The best $S^1$ is shown in blue, lying on the best $S^2$ in teal. In datasets 2 and 3, the axis
of the best $S^0$ (which consists of two antipodal points) is shown in black. 
In dataset 3, this also coincides with a standard mean of the data.
}
\end{figure}

%
%
%
%
%
%
%
%
%
%
%
%
%
%
%
%
%

\section{Grassmannians}
\label{sec:grass}
The Grassmannian  $G(k,n)$ of $k$-dimensional subspaces (or {\em $k$-planes}) of $\mathbb{R}^{n}$ is a symmetric space (see Sec. \ref{sec:symspace}). Data comprising subspaces may arise if we wish to track the eigenspace decomposition of  symmetric matrices such as diffusion tensors,
or if we collect a sequence of low-dimensional approximating subspaces
to Euclidean data using Euclidean PCA as some parameter (e.g. time) evolves.
 Related applications occur in computer vision and signal processing \cite{Gdata}. 


\def\Wspace{\mathsf{W}}
\def\Xspace{\mathsf{X}}
\def\Yspace{\mathsf{Y}}

The classification of geodesic submanifolds of Grassmannians is surprisingly complicated \cite{Klein}.
 Here we restrict our attention to a specific type of geodesic submanifold, namely the space of $k$-planes  orthogonal to a given subspace $\Wspace$ of $\mathbb{R}^n$.

\begin{lemma}
The space of $k$-planes in $\mathbb{R}^n$ orthogonal to a given $(n-m)$-dimensional subspace $\Wspace$
of $\mathbb{R}^n$ is a totally geodesic submanifold of $G(k,n)$ and is diffeomorphic to $G(k,m)$.
\end{lemma} 

\begin{proof}
The symmetric algebra has canonical decomposition $\mathfrak{g}=\mathfrak{h} + \mathfrak{m}$, where \cite[Ex. XI.10.3]{KN}
\begin{align*}
\mathfrak{g}&=\mathfrak{o}(n),\\
\mathfrak{h} &= \mathfrak{o}(k) +\mathfrak{o}(n-k)\\
&=
\left\{h(U,V):=
\begin{pmatrix}
U & 0 \\
0 & V
\end{pmatrix},
U\in\mathfrak{o}(k), V\in\mathfrak{o}(n-k)
\right\},
\intertext{and}
\mathfrak{m}&=\left\{m(X):=
\begin{pmatrix}
0 & -X\tran  \\
X & 0
\end{pmatrix},
X\in\mathbb{R}^{(n-k)\times k}
\right\}.
\end{align*}
 A geodesic connecting two points on a Grassmannian may be characterized as a linear interpolation of each principal angle. Fix an orthogonal basis $(e_1,\ldots,e_k)$ of the subspace at the origin, and extend this to an orthogonal basis $\{e_i\}$ of $\mathbb{R}^n$. Then any set of $k$ orthogonal vectors  orthogonal to $(e_1,\ldots,e_k)$ may be written as 
 \[
 \begin{pmatrix}
 0_{k\times k} \\ X\end{pmatrix}
 \]
 with respect to the basis $\{e_i\}$, where $X\in\mathbb{R}^{(n-k)\times k}$.
The geodesic connecting $X$ to the origin, intersecting $X$ at time $t=1$, is
$\exp(t m(X)).o$.  In particular, suppose $\Wspace$ is an $(n-m)$-dimensional subspace of $\mathbb{R}^n$ orthogonal to $(e_1,\dots,e_k)$. Then any orthogonal basis $(w_1,\ldots,w_{n-m})$ for $\Wspace$ may written as 
 \[
 \begin{pmatrix}
 0_{k\times (n-m)} \\ R\end{pmatrix}
 \]
 with respect to the basis $\{e_i\}$, where $R\in\mathbb{R}^{(n-k)\times (n-m)}$.
Then the geodesic  $\{\exp(t m(X)).o, t\in\mathbb{R}\}$ consists of subspaces orthogonal to $\Wspace$ if and only if $X\tran  R=0$. It therefore suffices to show that for any $R$, the subspace
\[
\mathfrak{m}' = \left\{
m(X), X\in\mathbb{R}^{(n-k)\times k}, X\tran  R = 0
\right\}
\]
defines a Lie triple system. A  calculation shows that
\[
[[m(X),m(Y)],m(Z)] = Z\tran  Y X\tran  - Z\tran  X Y\tran  - Y\tran  X Z\tran  + X\tran  Y Z\tran 
\]
and we are done as $X\tran  R=Y\tran  R= Z\tran  R = 0$ implies $(Z\tran  Y X\tran  - Z\tran  X Y\tran  - Y\tran  X Z\tran  + X\tran  Y Z\tran )R = 0$. 

Fixing an orthogonal basis of $\Wspace$, extending this to an orthogonal basis of $\mathbb{R}^n$, and expressing subspaces orthogonal to $\Wspace$ in terms of this basis gives the required diffeomorphism.
\end{proof}

Let the columns of the matrix $W\in\mathbb{R}^{n\times (n-m)}$ be
an orthonormal basis for the subspace $\Wspace$. Let $X$, $Y\in \mathbb{R}^{n\times k}$ be orthonormal bases for two elements $\Xspace$, $\Yspace$ of $G(k,n)$. 
The relationship between $\Xspace$ and $\Yspace$ is measured by their {\em principal angles} $\theta = (\theta_1,\dots,\theta_k)$,
defined by $\cos\theta_k = \sigma_k(X\tran Y)$. The geodesic distance between $\Xspace$ and $\Yspace$ in 
the Riemannian symmetric space $G(k,n)$ is $\|\theta\|_2$. Another popular measure of
distance between subspaces is $\max_k\theta_k$ \cite[p. 584]{golub}. However,
like Conway et al. \cite{conway}, we find that it is far easier and more natural to 
use the ``chordal distance'' $\|\sin\theta\|_2$ (so named because when equipped with this
metric, $G(k,n)$ isometrically embeds in a Euclidean sphere). The chordal and geodesic
metrics are compatible and thus have the same totally geodesic subspaces. 

In the present context, the advantage of the chordal distance is that it linearizes the calculation
of the distance from a $k$-plane to a totally geodesic submanifold.

\begin{lemma}
The chordal distance between two subspaces $\Xspace,\Yspace\in G(k,n)$ is given by  $||X\tran  Y^\perp||_F$, where $Y^\perp$ is the orthogonal complement of $Y$.
\end{lemma}
\begin{proof}
The squared chordal distance $d_c(\Xspace,\Yspace)^2$ is
$$\sum_{i=1}^k \sin^2\theta_i = k - \sum_{i=1}^k \cos^2\theta_i = 
 k - \sum_{i=1}^k \sigma_i^2(X\tran Y) = k - \|X\tran  Y\|_F^2.$$
The $n\times n$ matrix $Q=[Y | Y^\perp]$ is orthogonal, so we have
$$\|X\tran  Q\|_F^2 = \tr(Q\tran  X X\tran  Q) = \tr(X X\tran ) = k$$
and also
$$\|X\tran  Q\|_F^2 = \|X\tran [Y | Y^\perp]\|_F^2 = \|X\tran  Y\|_F^2 + \|X\tran  Y^\perp\|_F^2.$$
Therefore
$$ d_c(\Xspace,\Yspace)^2 = k - \|X\tran  Y\|_F^2 = \|X\tran  Y^\perp\|_F^2.$$
\end{proof}

An immediate consequence is the following
\begin{proposition}
The chordal distance from a subspace $\Xspace\in G(k,n)$ to the set $G(k,m)$ of $k$-planes
orthogonal to $W\in\mathbb{R}^{n,n-m}$ is $\|X\tran  W\|_F^2$.
\end{proposition}
\begin{proof}
We consider the cases $k=m$ and $k<m$ separately. If $k=m$ then $G(k,m)$ is a single point, $Y^\perp = W$, and we are done. If $k<m$, then 
an orthogonal basis for the orthogonal complement $\Yspace^\perp$ of any $\Yspace$ orthogonal to $\Wspace$ may be written as $[W|U]$ for some $U\in \mathbb{R}^{n,m-k}$ that satisfies $U\tran  W = 0$. We have 
$$d_c(\Xspace,\Yspace)^2 = \|X\tran  Y^\perp\|_F^2 = \|X\tran  W\|_F^2 + \|X\tran  U\|_F^2.$$
This is to be minimized over all choices of orthogonal $U$ that are also orthogonal to $W$. 
Any $U$ that is orthogonal to both $W$ and $X$ achieves $\|X\tran  U\|_F=0$, giving the result.
As we have
\begin{align*}
\dim(\Xspace\cup \Wspace)^\perp&=n - (\dim \Xspace + \dim \Wspace - \dim(\Xspace\cap \Wspace))\\
& = n-(k + (n-m) - \dim(\Xspace\cap \Wspace))\\
&= m-k+\dim(\Xspace\cap \Wspace) \\
&>0,
\end{align*}
we can choose such a $U$.
\end{proof}

As in the case of spheres,  the best approximating Grassmannians can now be read off from the SVD of a matrix representing from the dataset.

\begin{proposition}
Let $\Xspace_1,\dots,\Xspace_d$ be a set of $d$ k-planes  with 
orthogonal bases $X_1,\dots,X_d$.
Then the matrix $W$ minimizing the sum of squared chordal distances of the $\Xspace_i$ to ${\Wspace}$ is precisely the matrix of singular vectors of $X\tran $ corresponding to its $p$ smallest  singular values, where  $X=[X_1,\dots,X_d]\in\mathbb{R}^{n\times kd}$ is the matrix obtained by concatenating the $X_i$s.
The chordal distance of the $\Xspace_i$ to ${\Wspace}$ is the 2-norm of the $p$ smallest singular values of $X\tran $. The principal
symmetric space approximations are nested, in that the best $G(k,p)$ lies in the best $G(k,q)$ for
$p\le q$.
\end{proposition}

\begin{proof}
The sum of the squared chordal distances  is
\[
d(X,{\Wspace})^2=\sum_{i=1}^d d(X_i,{\Wspace})^2
= \sum_{i=1}^d \|X_i\tran W\|_F^2=\|X\tran W\|_F^2.
\]
This expression is formally identical to that studied in Proposition \ref{prop:s}, hence the result follows as in Proposition \ref{prop:s}.
\end{proof}

\section{Products of spheres}
\label{sec:products}
Given two symmetric spaces $(G,H,\sigma)$ and $(G',H',\sigma')$, the direct product $(G\times G',H\times H',\sigma\times\sigma')$ is also a symmetric space \cite[p. 228]{KN}. The simplest
case to consider is that of products of spheres, $S^{a_1}\times S^{a_2}\times\dots$. Here we focus on products of $n$ circles giving the $n$-torus (Sec. \ref{sec:tori}) and products of $n$ 2-spheres (Sec. \ref{sec:polysphere}). 

In the previous examples, of spheres and Grassmannians, the action of the symmetry group
was transitive. Our task was limited to selecting, from the single group orbit available, the
best point (or points). On products of spheres, the action of the symmetry group is not transitive; there are many (even infinitely many) distinct orbits. In fitting models with both continuous and discrete parameters, one common approach is to consider each value of the discrete parameters as specifying a different model; which value is chosen then corresponds to a model selection problem. This is the approach adopted here. We note that in the Bayesian paradigm model selection arises naturally through the choice of prior; however we will not pursue this further here.

Two examples illustrate the complexity of the situation. 

First, consider data consisting
of $n$ angles, i.e. $x\in\T^n$. Totally geodesic submanifolds are subtori
described by resonance relations of the form $a\cdot x = c$, $a\in\mathbb{Z}^n$, $c\in\mathbb{R}$. Each fixed $a$ specifies a different resonance relation, while the continuous parameter $c$ selects the best model for a given discrete parameter $a$.

Second, consider data consisting of $n$ points on $S^2$, i.e. spherical polygons. Totally geodesic manifolds are products of copies of $S^1$ and $S^2$. An example is $x_1$ lying on a great circle; $x_2$  lying on a second great circle, and obeying a resonance relation with $x_1$; $x_3$ arbitrary; and  $x_4,\dots,x_n$ being rotations of a fixed spherical polygon. 

\subsection{Tori}
\label{sec:tori}
\subsubsection{Classification of totally geodesic submanifolds of tori}
The product $(S^1)^n $ is a symmetric space that we identify with the flat torus $\T ^n := (\mathbb{R}/\mathbb{Z})^n$ with standard coordinates $x\in [0,1)^n$.
The connected  totally geodesic submanifolds of $\mathbb{R}^n$ are the
 affine subspaces; taking their translations by $\mathbb{Z}^n$ and passing to the quotient gives the connected totally
geodesic submanifolds of $\T ^n$. Amongst these, we wish to select those that are regular submanifolds.
We will describe them by the resonance relations that they satisfy.
The group $\T ^n$ acts by translations on $\T ^n$ and leaves the resonance relations invariant; we regard the submanifolds that satisfy different resonance relations as belonging to different models. Thus, the problem of finding the principal symmetric space approximations to given data involves first fixing the resonance relation and then
determining the best fitting submanifold that obeys that resonance relation.

We will show in Proposition \ref{prop:subtori} that the regular connected totally geodesic submanifolds of $\T ^n$ are all tori.
Up to translations, they are parameterized by unimodular matrices $A\in\mathbb{Z}^{k\times n}$,
i.e., matrices with integer entries all of whose $k\times k$ minors do not have a common factor
(their greatest common divisor is equal to 1). Specifically, they have the form
\begin{equation}
\label{eq:T}
T :=\{ x\in \T ^n\colon A x  = c\}
\end{equation}
for some $c\in[0,1)^n$. 

\begin{example}\rm
The case of geodesics in $\T ^2$ gives a feel for the requirement that the submanifold be represented by an unimodular matrix $A$. 
(i) The subset $x_1 + \sqrt{2}x_2=0$ of $\T ^2$, associated with $A=[1,\sqrt2]$, 
is totally geodesic, but it is an irregular submanifold. It is useless for data fitting as it passing arbitrarily close to every point of the torus.
(ii) The subset $2x_1=0$ of $\T ^2$, associated with $A=[2,0]$, consists of the two vertical lines $(0,y)$ and $(\frac{1}{2},y)$ for $0\le y<1$. This 
set is a regular totally geodesic submanifold, but it is not connected---and $A$ is not unimodular.
(iii) The subset $2x_1+5x_2=c$, associated with the unimodular matrix $A=[2,5]$, is
a regular, connected, totally geodesic submanifold of $\T ^2$. We will give
an example of fitting such a geodesic below.
\end{example}

\begin{proposition} \label{prop:subtori}\cite{lawrence} Every regular connected codimension-$k$ totally geodesic submanifold of $\T ^n$ is
a subtorus given by Eq. (\ref{eq:T}) for some $c\in[0,1)^k$ and unimodular $A\in \mathbb{Z}^{k\times n}$.
\end{proposition}
\begin{proof}
First let $A$ be unimodular. We
will show that $T$ in Eq. (\ref{eq:T}) is a regular connected codimension-$k$ totally geodesic submanifold and is a subtorus.
Rows can be added to $A$  to create a matrix, $C$, of determinant 1 \cite{smith}.
The linear map 
$$\phi\colon\R^n\to\R^n,\quad  \hat x\mapsto C\hat x$$ is invertible, therefore the map $\hat x \mapsto A \hat x  $ is surjective.
Let $x\in T$ and let $\hat x$ be any point in $\R^n$ such that $\hat x \mod 1 = x$.  
We are given that 
$A \hat x = c + m$ for some $m\in\mathbb{Z}^n$. From the surjectivity of $\hat x\mapsto A \hat x$, there
is a $p\in\mathbb{Z}^n$ such that $A p = m$. Therefore
$A(\hat x-p) = c$. That is, some integer translation of $\hat x$ lies on the affine subspace
$\{\hat x\in\R^n\colon A \hat x = c\}$, which is the cover of a connected totally geodesic submanifold of $\T^n$.
Hence $T$ is a totally geodesic submanifold of $\T ^n$.

The map $\phi$ descends to an automorphism of $\T ^n$; it  provides a change of coordinates on $\T ^n$. In coordinates $y = C x$, the submanifold
is given by $y_1 = c_1,\dots,y_k=c_k$, with $y_{k+1},\dots,y_n\in [0,1)^{n-k}$. This submanifold is a connected regular submanifold of $\T ^n$ and is a subtorus.

To show the converse, let $T$ be any regular connected totally geodesic submanifold of $\T ^n$.
Its translation $U$ to the origin is a subgroup, hence a subtorus of $\T ^n$. The kernel
of the exponential map of the Lie algebra of $U$ is a lattice in $\mathbb{Z}^n$. Form a matrix
whose rows are a basis of this lattice. The null space of this matrix has a unimodular integer basis whose entries
are the resonance relations satisfied by elements of $U$ and $T$. These form the rows of $A$.
\end{proof}

The matrix $A$ describes the resonance relations satisfied by the subtorus. If $a_i\cdot x \mod 1 = c_i$ for all $i$, then for any $m_i\in\mathbb{Z}$ we have 
$(\sum_{i=1}^k m_i a_i) \cdot x \mod 1 = c_i$ as well. That is, the set of resonance
relations forms a $k$-dimensional lattice $L$ in $\mathbb{Z}^n$, with the rows of $A$ as a basis.
Two matrices $A$, $A'$ describe the lattice, and the same
family of subtori, if there is a matrix $Z\in GL(k,\mathbb{Z})$ such that $A'=ZA$.

Recall that the {\em dual lattice} $L^*$ is defined by 
$$ L^* := \{y\in\mathbb{R}^n\colon y\in{\rm span}(L),\ A y \in \mathbb{Z}\}$$
For any $y\in L^*$ and $x\in T$, we have $A(x+y)  = c$, thus $x+y\in T$. 
Since span$(L)$ is orthogonal to the tangent space of $T$, the lifted subtorus is the product 
of an affine space and the dual lattice $L^*$.

\begin{example}\label{ex:2torus}
 Let $n=3$ and $k=2$. We consider the 1-dimensional subtorus (i.e. geodesic)
that passes through 0 in direction $d=[1,2,3]\tran$. The transposed null space of $d$ is
$$A := \begin{bmatrix}-3 & 0 & 1 \\ -2 & 1 & 0\end{bmatrix}.$$
That is, the points on the subtorus are those that satisfy the resonance relations
$-3 x + z = 0$ and $-2 x + y = 0$.
A basis for the dual lattice is given by
$$ B := A\tran (A A\tran )^{-1} = \begin{bmatrix}-\frac{3}{14} & -\frac{1}{7}\\-\frac{6}{14} & \frac{5}{7} \\ 
\frac{5}{14} & -\frac{3}{7} \end{bmatrix}.$$
The lattice generated by the columns of $B$ shows the intersection between the subtorus and
a lifted plane orthogonal to it (see Figure \ref{fig:t123}).
\end{example}

\begin{figure}
\begin{center}
\includegraphics[width=5cm]{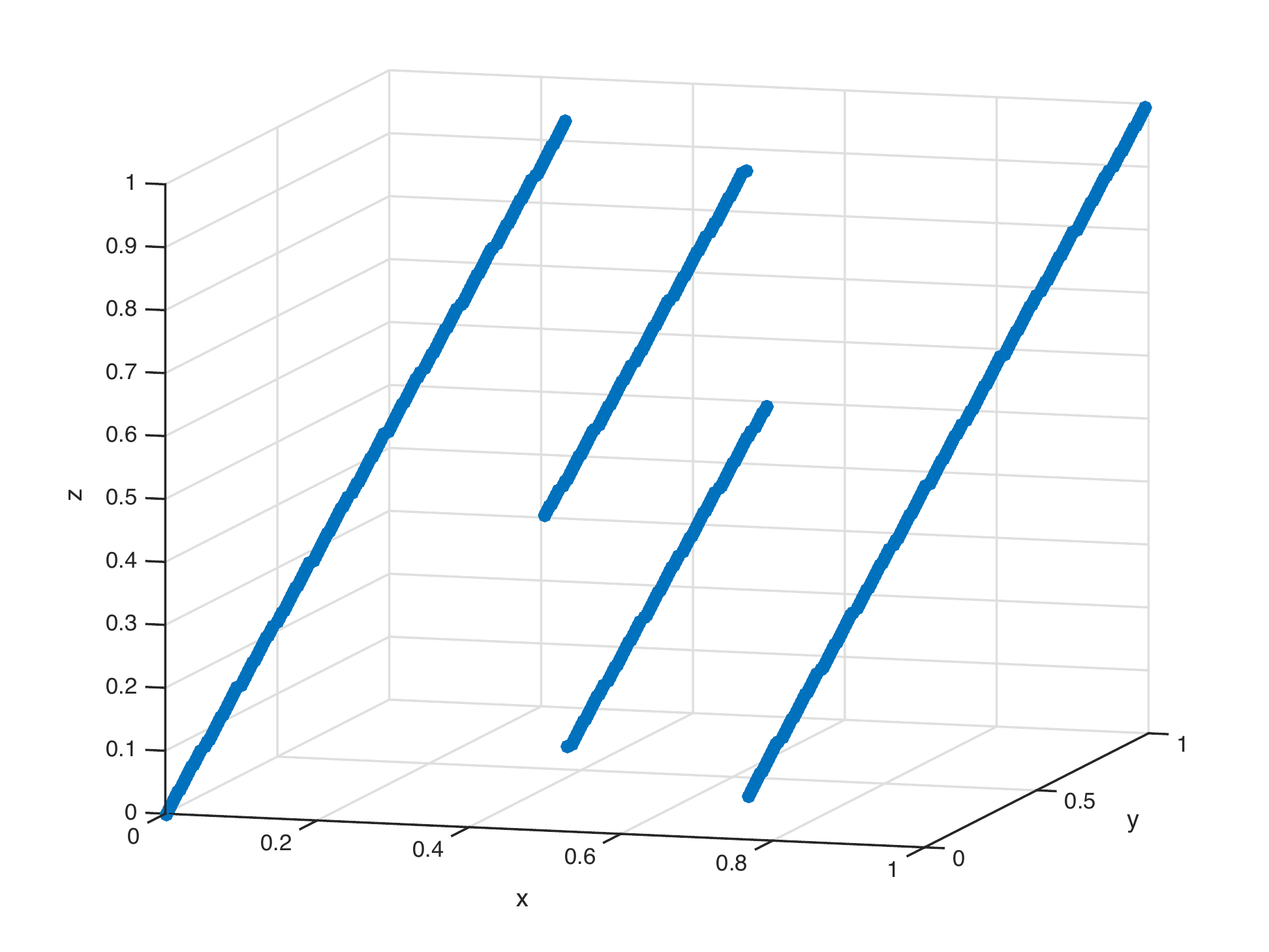}
\hspace{5mm}
\includegraphics[width=5cm]{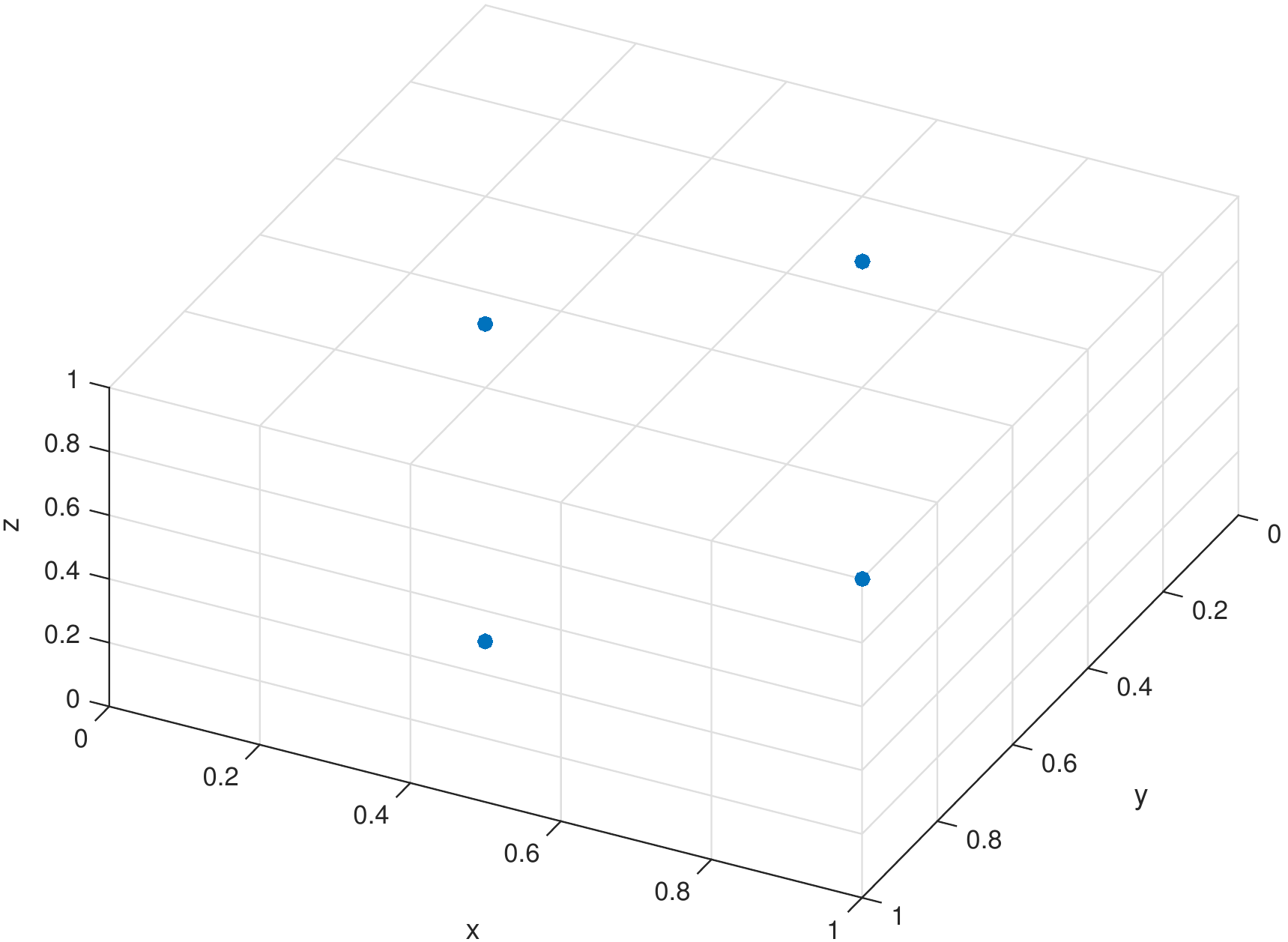}\\
\end{center}
\caption{\label{fig:t123} The geodesic in $\T ^3$ through the origin 
in direction $d=[1,2,3]\tran$, seen from two different viewpoints.
Viewing in direction $d$ (right) shows the lattice formed in $\mathbb{R}^3$ by the intersection
of the lifted geodesic with an orthogonal plane.}
\end{figure}

\begin{example}Let $n=2$, $k=1$, $c=0$, and $A=[2,5]$. A basis for the dual lattice is
$A\tran(A A\tran)^{-1}=[\frac{2}{29};\frac{5}{29}]$; this vector is orthogonal to the tangent space of the geodesic and gives the spacing between its successive winds, which are spaced a distance
$1/\sqrt{29}\approx0.19$ apart (see Figure \ref{fig:torus}). The fractional part of $2 x_1 + 5 x_2$ measures the angular distance from a point $x$ to the geodesic.
\end{example}

\subsubsection{Finding the best subtorus with given resonance relation}
We now consider the problem of computing the distance from a datapoint to a subtorus.
Consider the example shown in Figure \ref{fig:t123}. 
To compute the Euclidean distance, 
it is necessary to (i) lift
the datapoint to $\R^3$; (ii) project to a plane orthogonal to the tangent space of the subtorus;
and (iii) find the nearest point in the dual lattice $L^*$; and (iv) compute the distance
to this point. The difficult step is (iii), an instance of the Closest Vector Problem (CVP) in 
the dual lattice $L^*$. However, this is a difficult problem in high dimensions and
the degree of complexity it entails seems unnecessary here.

This step can be avoided by modifying the metric suitably. 
 
As we are working with angular distances, we replace the standard angular distance
$d(x,y)=2\pi |x-y|\le\pi$, $x,y\in\T $,
by the chordal distance $d_c(x,y)= \frac{1}{2}\sin\pi|x-y|\le\frac{1}{2}$. The Karcher mean of
angles $x_i\in[0,1)$ is easily calculated as the circular mean
$$ \bar x = \mathrm{atan2}\left(\frac{1}{d}\sum_{i=1}^d \sin(2\pi x_i),\frac{1}{d}\sum_{i=1}^d \cos(2\pi x_i)\right).$$
We define the circular mean $\bar x$ of $x_i\in\T^n$ componentwise.
We now introduce a further modification of the metric that is adapted to the chosen family
of subtori.

\begin{definition} Let $C\in GL(n,\mathbb{Z})$. Then $d_C(x,y) := d_c(Cx,Cy)$.
\end{definition}

\begin{proposition}
Let $A$ be the first $k$ rows of $C\in GL(n,\mathbb{Z})$.
Amongst the subtori with resonance relation $A$, the subtorus of best fit in the metric $d_C$
to the data $x_1,\dots,x_d\in \T ^n$ is $\{x\in\T^n\colon A x  = c\}$,
where $c$ is the circular mean of $A x_1,\dots,Ax_d$.
\end{proposition}

\begin{proof}
In coordinates $y=Cx$, the subtorus is given by 
$$\hat T=\{y\in\T^n\colon y_i=c_i,\ i=1,\dots,k\},$$
 and the distance from $y$ to the subtorus $T=\{x\colon Ax=c\}$ is determined by the angular displacement $\hat y-c\in\T^k$ where $\hat y = (y_1,\dots,y_k)$.
The distance $d_C(x,T)=d_c(\hat y,c)$ is minimized at $c = \bar{\hat y}$.
\end{proof}

\begin{figure}
\begin{center}
\includegraphics[width=6cm]{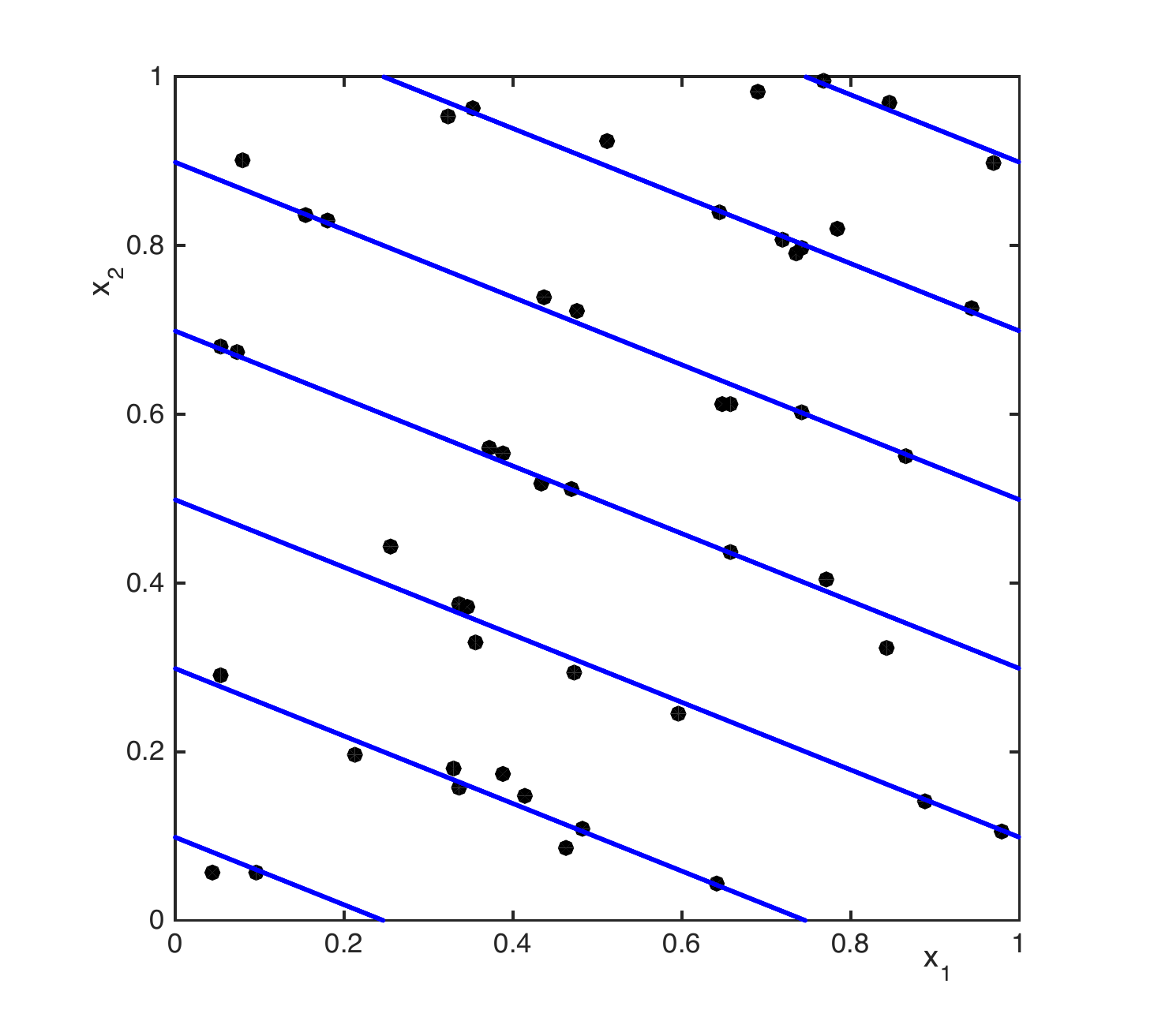}
\raise10mm\hbox{\includegraphics[width=7cm]{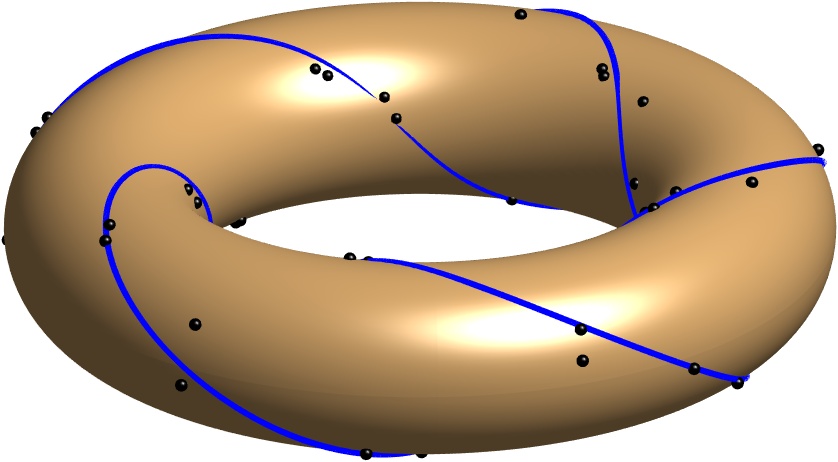}}
\end{center}
\caption{\label{fig:torus}
Fitting data on a torus. Here the closed geodesic of best fit is computed to a set of 50 data points
on $S^1\times S^1$.
The data set is synthetic and has been chosen to lie near the geodesic with resonance relation
$2 x_1 + 5 x_2 = $ const.;
each data point has normal random noise of standard deviation $0.1/(2\pi)$ in each angle. }
\end{figure}

Note that although $d_C$ depends on the whole matrix $C$, the best subtorus only
depends on its first $k$ rows, $A$.

If the rows  of $A$ are pairwise orthogonal and all have the same length,  then $d_C(x,T)=d_c(x,T)$, but in general the two metrics are not the same. Different bases $A$ of $L$ lead to different distance measures and different best tori. Most lattices have no orthogonal bases.
However, it does point
to the necessity of choosing a good basis for the resonance relations, one
in which the relations are as nearly orthogonal as possible. This is
another standard problem in lattice theory, one that can be solved exactly
in low dimensions, and approximately (by the LLL algorithm) in high dimensions.

\medskip
\noindent
{\bf Example 3 (ctd.)}
 The angle between the two basis vectors in Example
1 is $32^\circ$. A more nearly orthogonal basis is 
$$\begin{bmatrix}-1 & -2 \\ -1 & 1 \\ 1 & 0\end{bmatrix}$$
in which the angle between the basis vectors is $75^\circ$.
\medskip

The best fit of a 1-torus with fixed resonance relation to data in $\T^2$ is shown
in Figure \ref{fig:torus}, and the best fits of 1-tori and 2-tori with fixed resonance
relations is shown in Figure \ref{fig:2torus}.

\subsubsection{Model selection for tori}\label{model_selection}

In finding the best subtorus amongst those with {\em fixed} resonance relations, 
the overall scaling of the metric is irrelevant. It becomes relevant during the model
selection phase, when fitted subtori with different resonance relations are compared.
Here we illustrate one possible approach to this issue using
(i) the unscaled circular means, as given above; and (ii)
the `leave one out' model selection method.
Item (i) means that the maximum distance of any point to a subtorus, in each 
coordinate, is 0.5, regardless of the resonance relations or winding density of the 
subtorus. While the metric could be scaled down, to make it more closely approximate
the original Riemannian distance in $\T ^n$, doing so would strongly
favour models with very dense windings, as they pass close to every point
in $\T ^n$. Therefore we stick with the unscaled metric $d_C$ defined above.

Item (ii) means that
for each data point $i$, the subtorus of best fit to the data set
omitting point $i$ is calculated, from which the prediction error of this fit to data point $i$
can be calculated. In the scaled chordal metric we are using, this is $e_i := \|\frac{1}{2}\sin\pi(A  x_i-c)\|_2$
This lies in the interval $[0,\frac{1}{2}\sqrt{k}]$, taking the value 0 if the omitted data point lies on the geodesic,
and taking the value $\frac{1}{2}\sqrt{k}$ if it lies midway between two winds of the geodesic in each of the
$k$ directions.
The mean projection error $\|e_i\|_2$ is taken as a measure of the goodness of fit of the model with
resonance relations $A$. The resonance relation with minimum $\|e_i\|_2$ is chosen.

The method is illustrated on a synthetic data set of 50 points that lie near the geodesic
$2 x_1 + 5 x_2 = $ const.; see Figure \ref{fig:torus}. All resonance relations
with $\|A\|_\infty<10$ are tested. The leave-one-out method selects the `correct'
$A =[2,5]$ for this dataset.

\subsubsection{Nested approximations}

So far we have presented a method for finding the best subtorus of a given dimension.
However, note that the same method naturally produces a nested sequence
of approximations of subtori of different dimensions.

\begin{proposition}
Let $A\in GL(n,\mathbb{Z})$, let $A_k$ be the first $k$ rows of $A$, and let $x_1,\dots,x_d$ be data in $\T ^n$.
For each $k=1,2,\dots,n-2$, the $(n-k)$-dimensional subtorus with resonance relations
$A_k$ of best fit  contains the
$(n-k-1)$-dimensional subtorus of resonance relations
$A_{k+1}$ of best fit.
\end{proposition}
\begin{proof}
The subtori are $A_k  x  = c$, where $A_k$ is the first $k$ rows of $A$
and the entries in $c\in\mathbb{R}^k$ are the circular means of $A x_i$. 
Adding another resonance relation, i.e. increasing $k$ by 1,
does not change the first $k$ entries of $c$.
\end{proof}

Thus to each  $A\in GL(n,\mathbb{Z})$ we get a nested
sequence of subtori of dimension 1 to $n-1$ and an approximation
error associated to each subtorus. If the rows of $A$ are 
nearly orthogonal, this is a close analogue of standard PCA.

\medskip
\noindent
{\bf Example 3 (ctd.)}
We take
$$A=\begin{bmatrix}-1 & -2 & 0\\ -1 & 1 & 1 \\ 1 & 0 & -1 \end{bmatrix}.$$
We take a synthetic data set of 50 points on $\T ^3$. 
(See Figure \ref{fig:2torus}.)
When $k=1$ we are seeking the best 2-torus of the form
$-x -y - z=$ const.; it has mean error 0.049.
When $k=2$ we are seeking the best 1-torus of the
form $-x-y-z=$ const., $-2 x + y = $ const., i.e., the best
geodesic parallel to $[1,2,3]$. It has mean error 0.049 orthogonal
to the previously found 2-torus, i.e. in the direction $[-1,-1,1]$,
 and mean error 0.169 in the direction $[-2,1,0]$; its mean
 error is $\sqrt{0.049^2+0.169^2}=0.176$. 
 These errors are scaled so that the distance
 between winds is 1, i.e., the distance
 between winds of the blue 2-torus is 1 and
 the distance, measured within the blue 2-torus,
 between winds of the red 1-torus is 1.
 \medskip
 
 \begin{figure}
 \begin{center}
 \includegraphics[width=6.5cm]{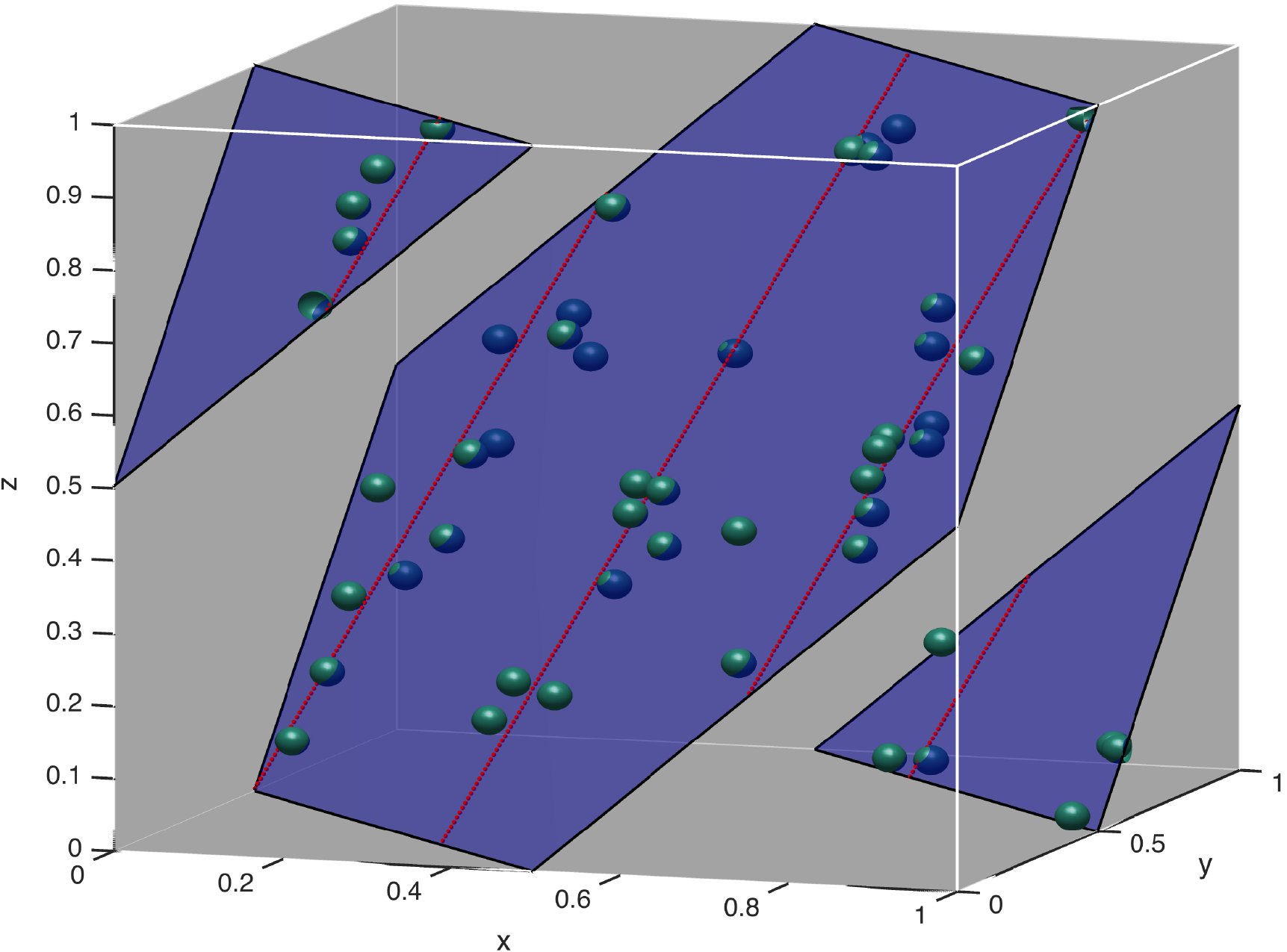}
 \hspace{4mm}
 \includegraphics[width=7.5cm]{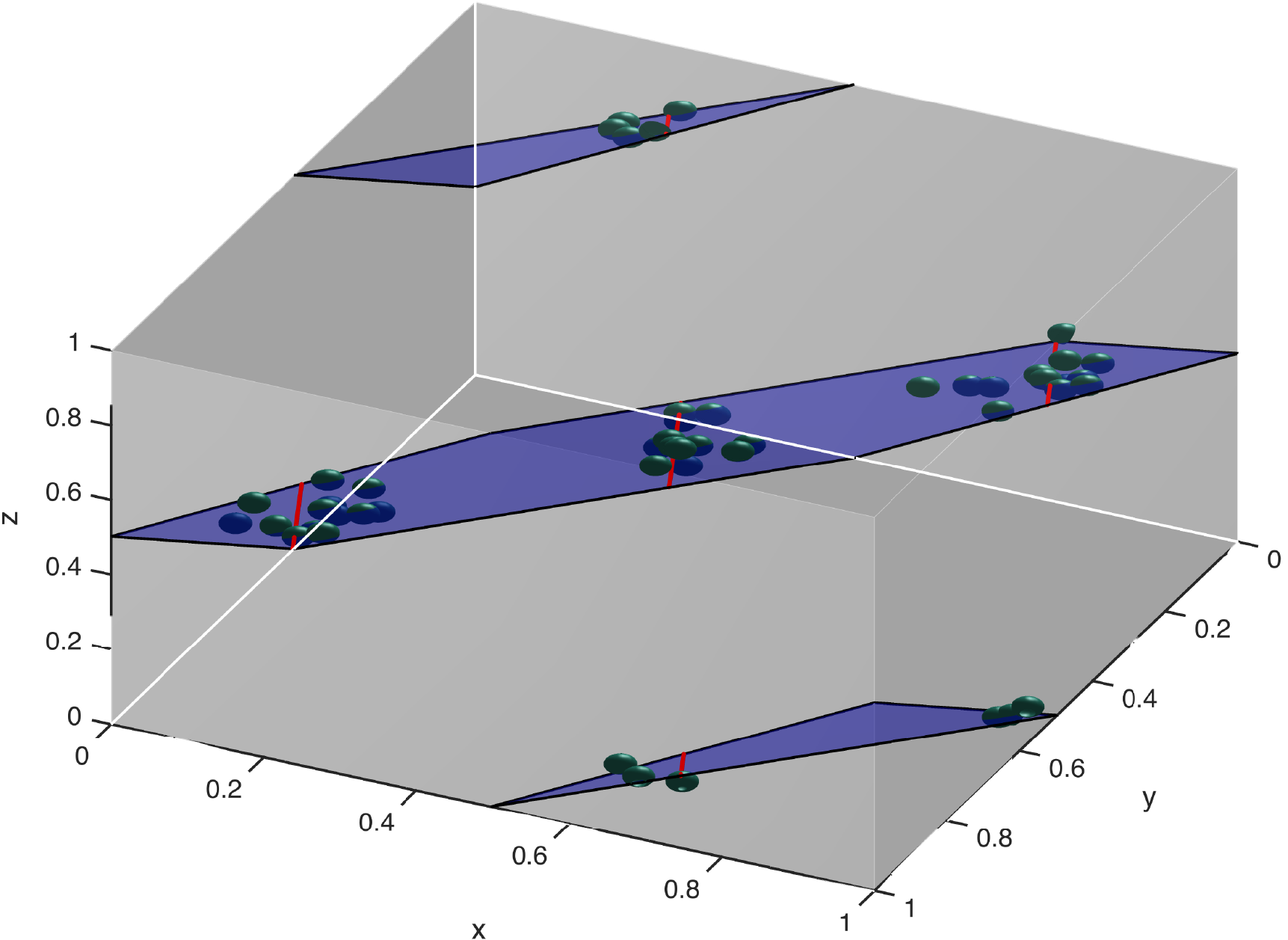}\\
 \end{center}
 \caption{\label{fig:2torus} Fitting data on a torus (see Example \ref{ex:2torus}). The best
 1-torus and 2-torus approximating 50 data points on $\T ^3$,
 are shown from two different viewing directions. The
 subtori have been chosen from those with fixed resonance
 relations, i.e., only their translations fitted.}
 \end{figure}

\subsection{Polyspheres}
\label{sec:polysphere}

The polyspheres $S^2 \times\cdots\times S^2$ arise frequently in practical applications, for example in joint data. We begin by considering the case $S^2\times S^2$, as the arguments are analogous in higher dimensions. We must classify the geodesic submanifolds of $S^2\times S^2$, for which purpose  we recall that the symmetric algebra of $S^2$ is $\mathfrak{o}(3)=\mathfrak{h}+\mathfrak{m}$, where
\[
\mathfrak{h} =
\left\{h(A):=
\begin{pmatrix}
0 & 0 \\
0 & A
\end{pmatrix}\!,
A\in\mathfrak{o}(2)\right\},
\quad
\mathfrak{m} = \left\{
m(\xi):=
\begin{pmatrix}
0 & -\xi\tran  \\
\xi & 0
\end{pmatrix}, \xi\in\mathbb{R}^2
\right\}.
\]

A straightforward calculation shows that 
$[m(\xi),m(\zeta)]=h\big([[\xi,\zeta]]\big)$, where $[[\xi,\zeta]]:=\zeta\xi\tran  - \xi\zeta\tran$, and that $[h(A),m(\xi)] = m(A\xi)$, which will be used frequently in the coming calculations. The symmetric algebra of $S^2\times S^2$ is $(\mathfrak{h}+\mathfrak{h})+(\mathfrak{m}+\mathfrak{m})$, and the totally geodesic submanifolds (and hence symmetric subspaces) take the form
\[
\{ \exp(v)\cdot x \;|\; v\in \mathfrak{m'} \},
\]
where $x\in S^2\times S^2$ is an arbitrary point, and $\mathfrak{m'}\subset \mathfrak{m}+\mathfrak{m}$ is a Lie triple system. We now state the results we obtain, the proofs of which are to be found in the appendix. We first classify the geodesic submanifolds of $S^2\times S^2$:
\begin{theorem}\label{2d}\ \\
The 1-dimensional connected geodesic submanifolds of $S^2\times S^2$ are the submanifolds
$\{(\exp(r_1 t)x,\exp(r_2 t)y\colon t\in\R\}$, where $r_i\in\mathfrak{o}(3)$ and $x$, $y\in S^2$.\\
The 2-dimensional connected geodesic submanifolds of $S^2\times S^2$ are of the following types:
\begin{enumerate}
\item $\{(x,Rx)\colon x\in S^2\}\simeq S^2$, where $R\in O(3)$ is fixed.
\item $\{(x,y)\colon x\in S^2\}\simeq S^2$ where $y\in S^2$ is fixed, 
or \\$\{(x,y)\colon y\in S^2\}\simeq S^2,$ where $x\in S^2$ is fixed.
\item $\{(\exp(r_1 t_1)x,\exp(r_2 t_2)y)\colon t_1,t_2\in\R\}\simeq S^1\times S^1$, where $r_i\in\mathfrak{o}(3)$ and $x$, $y\in S^2$ are fixed.
\end{enumerate}
The 3-dimensional connected geodesic submanifolds of $S^2\times S^2$ are precisely the submanifolds $\{(x,\exp(rt)y)\colon x\in S^2,\ t\in\R\}\simeq S^2\times S^1$, for some fixed $y\in S^2$ and $r\in\mathfrak{o}(3)$.
\end{theorem}

The principal symmetric space decompositions of $S^2\times S^2$ can by summarized by the following diagram: 
\begin{diagram}
& & S^2\times S^2 & & & & \\
& \ldTo{\hbox{\rm Type 1}} & & \rdTo & & & &\\
S^2 & & & & S^2 \times S^1 & &  \\
\dTo & & & \ldTo{\hbox{\rm Type 2}} &  & \rdTo{\hbox{\rm Type 3}} & \\
 S^1 & & S^2 & & & & S^1\times S^1 \\
  & & \dTo & & & & \dTo \\
  & & S^1 & & & & S^1 
\end{diagram}

\begin{theorem}\label{high poly}
The totally geodesic submanifolds of $(S^2)^n$ are isomorphic to a product of copies of $S^2$ and $S^1$. The rooted tree structure of the principal symmetric space approximations $\mathcal{PSSA}(X,(S^2)^n)$ is characterized as follows: every edge of the tree corresponds to either 
\begin{enumerate}
\item A reduction to an $m$-torus inside an $n$-torus ($m<n$)
\item A 2-dimensional reduction arising from a coupling of two spheres after situation 1 of Lemma~\ref{2d}.
\item One of the following one-dimensional reductions: the restriction to a great circle $S^1\subset S^2$, or to a trivial submanifold ${x_0}\subset S^1$.
\end{enumerate}
\end{theorem}

\begin{figure}
	\begin{center}
		\includegraphics[width=0.4\textwidth]{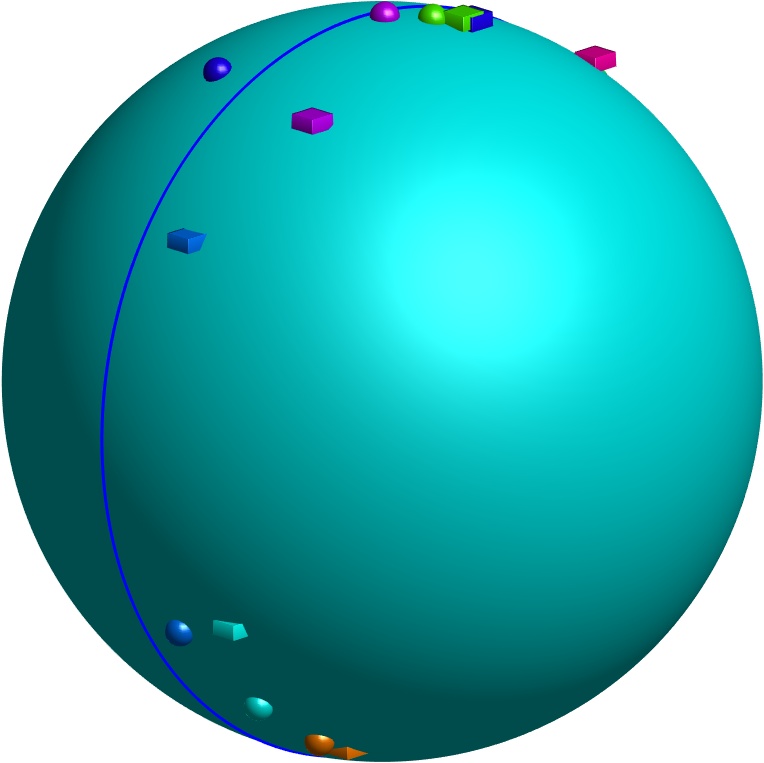}
	\end{center}
	\caption{\label{fig:2poly} Fitting data on a polysphere (see Example \ref{2p}). The rotations of points $x_i$ are plotted as cubes, whilst the points $y_i$ are plotted as circles; a different colour is chosen for each $i$. The great circle shows the best subspace $S^1\subset S^2\subset S^2\times S^2$.}
\end{figure}

The complexity of rooted tree arising from principal symmetric space approximations of polyspheres shows that a model selection problem cannot be avoided; see the remarks in the introduction and \S \ref{model_selection}.

\begin{example}\label{2p}
	We illustrate the behaviour with two synthetic datasets on $S^2\times S^2$. The data of \ref{fig:2poly} illustrates the middle branch of the tree, whilst the rightmost branch of the tree is shown in Figure \ref{fig:2tpoly}.
\end{example}

\begin{figure}
	\begin{center}
		\includegraphics[width=0.3\textwidth]{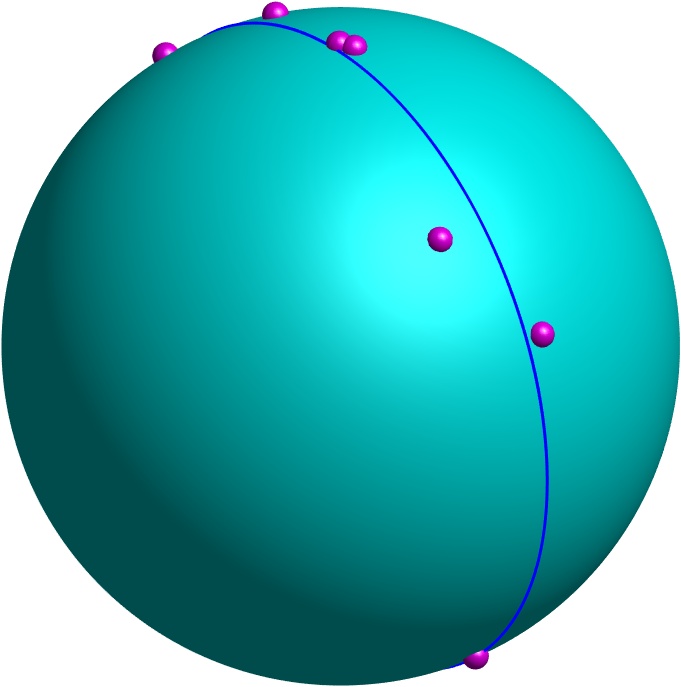}
		\includegraphics[width=0.3\textwidth]{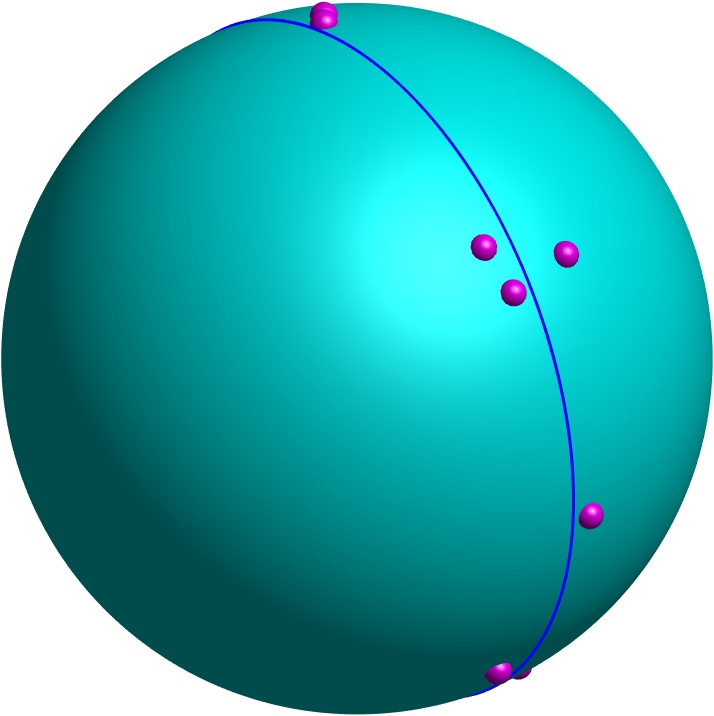}
		\includegraphics[width=0.35\textwidth]{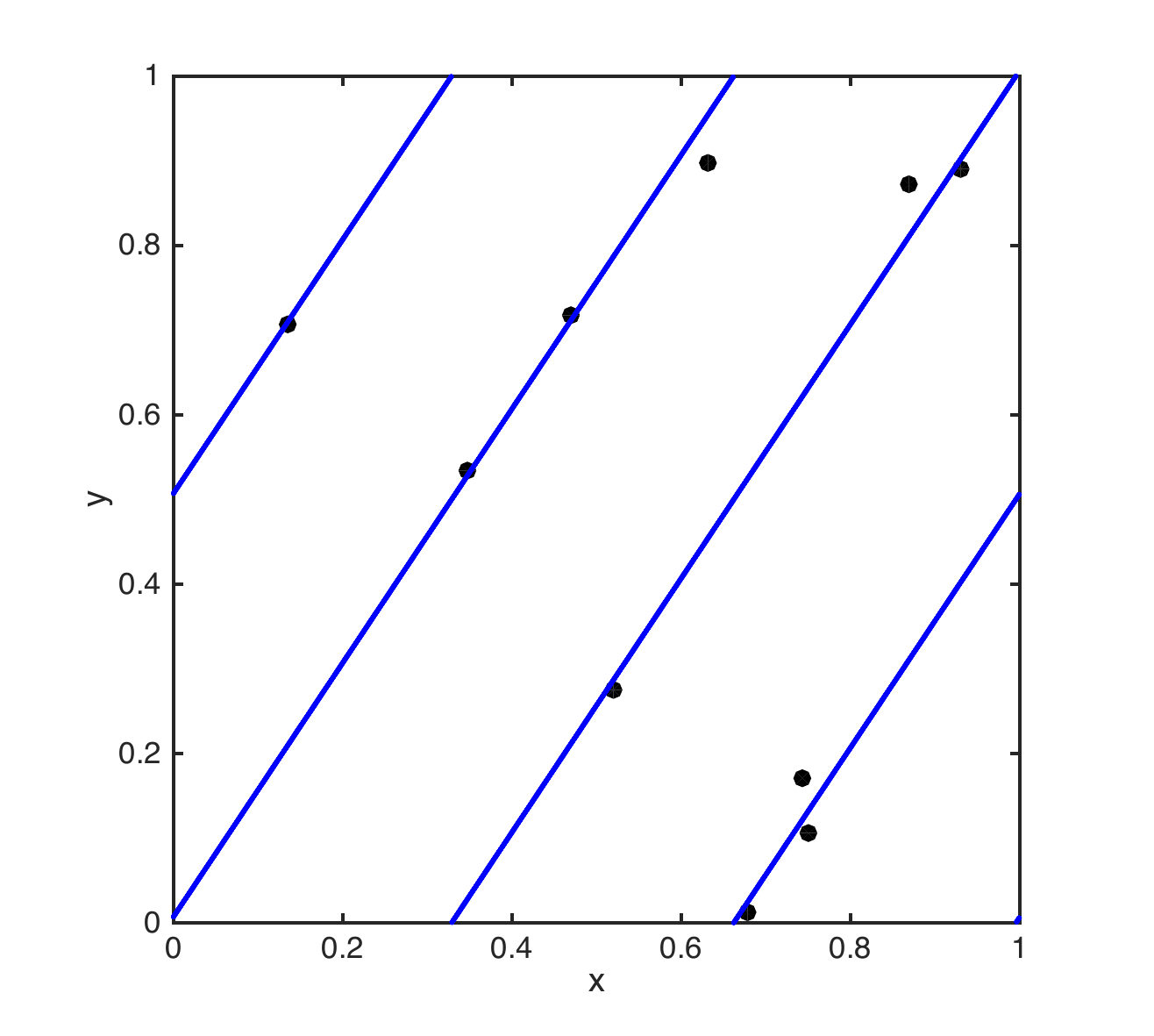}
	\end{center}
	\caption{\label{fig:2tpoly} Fitting data on a polysphere (see Example \ref{sp}). The best approximating torus $S^1\times S^1\subset S^2\times S^2$ is shown with two great circles inside the two spheres. Due to the difficulty of plotting the nested approximation $S^1\subset S^2\times S^2$ we have plotted (right) the projection of the points in $S^2\times S^2$ to the approximating torus $S^1\times S^1$ and shown the best approximating $S^1$ as a subset of this.}
\end{figure}

\appendix

\small

\section{Proofs for polyspheres}

We collect here the proofs of the classification results for geodesic submanifolds of polyspheres. We begin by stating some useful lemmas.

\begin{lemma}\label{subspace}
	The two-dimensional vector subspaces of $\mathbb{R}^2 \oplus \mathbb{R}^2$ take one of the forms (up to a reordering of basis elements):
	\begin{enumerate}
		\item 	\[
		\{ (\xi, A\xi) \;|\; \xi \in\mathbb{R}^2 \},\quad A\in\mathrm{Mat}(2\times 2)
		\]
		\item \[
		\{ (t_1\zeta_1, t_2\zeta_2) \;|\; t_i\in\mathbb{R} \},\quad \zeta_i\in\mathbb{R}^2
		\]
	\end{enumerate}
	The three-dimensional vector subspaces of $\mathbb{R}^2 \oplus \mathbb{R}^2$ take the form (up to a reordering of basis elements):
	\[
	\{ (\xi, A\xi + t\zeta) \;|\; \xi\in\mathbb{R}^2,t\in\mathbb{R} \},\quad \zeta \in\mathbb{R}^2, A\in\mathrm{Mat}(2\times 2)
	\]
\end{lemma}

\begin{proof}
	Suppose the subspace $V$ is spanned by two vectors, $(u_1,v_1),(u_2,v_2)$. We first consider the case where $u_1$ and $u_2$ are linearly independent. Then if $(x,y)\in V$ we have $x=a u_1 + b u_2$, where 
	\[
	\begin{pmatrix}
	a \\
	b
	\end{pmatrix}
	= \Big( u_1 \big| u_2 \Big)^{-1} x,
	\]
	and hence $y= a v_1 + b v_2$, i.e.
	\[
	y = \Big( v_1 \big| v_2 \Big)\Big( u_1 \big| u_2 \Big)^{-1} x,
	\]
	and the subspace takes the form $(x,Bx)$ for some matrix $B$. If the $u_i$ are linearly dependent, but the $v_i$ are linearly independent, the argument is similar. The remaining case to consider is where both the $u_i$ and $v_i$ are linearly dependent, in which case the subspace takes the form $(t_1 u, t_2 v)$, for some fixed vectors $u$ and $v$.
	
	Now suppose the subspace $V$ is generated by $(u_1,v_1),(u_2,v_2),(u_3,v_3)$. Then either the $(u_i)$ or $(v_i)$ span $\mathbb{R}^2$, assume that $(u_i)$ do. For any fixed $t$, if $(x,y)\in V$ then $x=a u_1 + b u_2 + t u_3$ in a unique manner, indeed
	\[
	\begin{pmatrix}
	a \\
	b
	\end{pmatrix}
	= \Big( u_1 \big| u_2 \Big)^{-1} (x-tu_3).
	\] 
	Then
	\[
	y = \Big( v_1 \big| v_2 \Big)\Big( u_1 \big| u_2 \Big)^{-1} x + t \Big( v_3 - \Big( v_1 \big| v_2 \Big)\Big( u_1 \big| u_2 \Big)^{-1} u_3 \Big)
	\]
	and we see that $V$ takes the form $(x,Ax + tv)$ for some fixed $A$ and $v$.
\end{proof}

We will also make use of the following elementary results.
\begin{lemma}\label{matrices}
	Let $A\in\mathfrak{so}(2)$ be non-zero, and let $B$ be a $2\times 2$ matrix. Suppose that $BAB^T Bz = BAz$ for all $z\in\mathbb{R}^2$. Then either $B\in O(2)$ or $B=0$.
\end{lemma}

\begin{proof}
	Firstly note that the condition $BAB^T Bz = BAz$ for all $z\in\mathbb{R}^2$ is equivalent to $BAB^T B = BA$, by for instance the bijective correspondence between linear transformations and matrices. Let $C = B^T B$; multiplying both sides by $B^T$ results in the relation $CAC = CA$. Note that $C$ is symmetric, whilst $A$ is antisymmetric, and hence $CAC$ is also antisymmetric; as $\mathfrak{so}(n)$ is one-dimensional we have $CAC = kA$ for some $k\in\mathbb{R}$. This gives also $CA = kA$, and taking transposes $-AC = -kA$. Then $CAC = C(AC) = kCA = k^2 A$. It follows that $k^2 A = k A$, and hence $k = 0$ or $1$. Moreover, as  $A$ is invertible, we conclude that $C = kI$ and the result follows.
\end{proof}

\begin{lemma}\label{exp}
	Let $V=\{ m(\xi, A\xi) \;|\; \xi \in\mathbb{R}^2 \}$, where $A\in\mathrm{Mat}(2\times 2)$. Then the submanifold $\{ \exp(v)\cdot x \;|\; v\in V \}$ takes the form $\{(x,Rx)\in P\}\simeq S^2$, for some $R\in O(3)$
\end{lemma}

\begin{proof}
	Without loss of generality we consider a basis such that $x=(e_z,e_z)$, in which
	\[
	\exp(m(A\xi))\cdot e_z
	\]
	for some $A\in O(2)$, where $e_z=(0,0,1)^T$. Note that $m(A\xi)=Bm(\xi)B^T$, where $B\in O(3)$ takes the form
	\[
	B = \begin{pmatrix}
	A & 0 \\
	0 & 1
	\end{pmatrix}
	\] 
	Then
	\[
	\exp(m(A\xi))=\exp(Bm(\xi)B^T) = B\exp(m(\xi))B^T,
	\]
	and as $B^T e_z = e_z$, we have that
	\[
	\exp(m(A\xi))\cdot e_z = B \exp(m(\xi))\cdot e_z,
	\]
	as $\{\exp(m(\xi))\cdot e_z\}$ spans $S^2$ the result follows immediately.
\end{proof}

\begin{lemma}\label{aux mat}
	Suppose that $D\in O(2)$, and that $D^T xy^T = xy^T D$ for all $x,y\in\mathbb{R}^2$. Then $D=I$.
\end{lemma}

\begin{proof}
	Let $C_i=x_i y_i^T$, with $x_1=y_1=(1,0)^T, x_2=y_2=(0,1)^T, x_3=(1,0)^T, y_3=(0,1)$. Writing out 
	\[
	D=\begin{pmatrix}
	a & b \\
	c & d
	\end{pmatrix}
	\]
	we see that the equations $D^T C_i = C_i D$ give respectively $b=0$, $c=0$ and $a=d$, from which the result follows.
\end{proof}

\noindent\textbf{Proof of Theorem~\ref{2d}}

\begin{proof}
	We must search for and exponentiate Lie triple systems $\mathfrak{m'}\subset \mathfrak{m}+\mathfrak{m}$; these take the form $m(V)$, where $V\subset \mathbb{R}^2 \oplus \mathbb{R}^2$ take the forms described in Lemma~\ref{subspace}. Amongst the two-dimensional cases, we begin by those of the form $V = \{ (\xi, B\xi) \;|\; \xi \in\mathbb{R}^2 \},\quad B\in\mathrm{Mat}(2\times 2)$, whereupon we compute
	\[
	[[m(x,Bx),m(y,By)],m(z,Bz)] = m\big( [[x,y]]z, B[[x,y]]B\tran  Bz\big).
	\]
	As $[[x,y]]\in\mathfrak{o}(2)$, Lemma~\ref{matrices} shows that we obtain a Lie triple system if either $B\tran  B=I$, or $B=0$. We see by Lemma~\ref{exp} that taking an orthogonal $B$ results in a subspace of the first kind listed, whilst taking $B=0$ trivially results in the second kind. It remains to check subspaces $\{ (t_1\zeta_1, t_2\zeta_2) \;|\; t_i\in\mathbb{R} \},\quad \zeta_i\in\mathbb{R}^2$, but these are trivially totally geodesic, as then $[\mathfrak{m'},\mathfrak{m'}]=0$. Exponentiating the resulting subspace gives the third case of the lemma.
	
	We then consider the three-dimensional submanifolds,  where $V=(\xi, A\xi + t\zeta) \;|\; \xi\in\mathbb{R}^2,t\in\mathbb{R} \},\quad \zeta \in\mathbb{R}^2, A\in\mathrm{Mat}(2\times 2)$. We first compute
	\[
	[m(Bx + sv), m(By + tv)] = h\big( B[[x,y]]B\tran  + v(tx\tran  - sy\tran )B\tran  + B(sy-tx)v\tran  \big)
	\]
	It follows that
	\begin{align*}
	[[m(Bx+sv),m(By+tv)],m(Bz+rv)] &= \\
	m\big( B[[x,y]]B\tran  Bz &+ v(tx\tran  - sy\tran )B\tran  Bz + B(sy-tx)v\tran  Bz \\ 
	+ rB[[x,y]]B\tran  v &+ rv(tx\tran  - sy\tran )A\tran  v + rA(sy-tx)v\tran  v \big)
	\end{align*}
	To obtain a Lie triple system, we require that the right hand side equal $m(B[[x,y]]z + pv)$, for some scalar $p$, for all $x,y,z$. This is only possible if $B=0$.
	Indeed, setting $s,t=0$ results in the equation
	\[
	B[[x,y]]B\tran  Bz + rB[[x,y]]B\tran  v = B[[x,y]]z + pv,
	\]
	fix $x,y$ so that $[[x,y]]=A$ is an arbitrary (non-zero) matrix in $\mathfrak{o}(2)$. We are left with two lines in $\mathbb{R}^2$ (as we vary $r,p$), it follows that we require $v$ to be an eigenvector of $BAB\tran$, however as $BAB\tran\in\mathfrak{o}(2)$ it has no real eigenvectors unless $B=0$. The results then follows immediately upon exponentiating $m(V)$.
\end{proof}

\noindent\textbf{Proof of Theorem~\ref{high poly}}

\begin{proof}
	We begin by noting that the case $S^2\times S^2$ follows this pattern: the edge $1$ is the two-dimensional reduction resulting from a coupling of spheres, the edge $2$ comes from the inclusion ${x_0}\subset S^1$ , and the edge $3$ is an inclusion $S^1\subset S^2$. The remaining edges are clearly also of this pattern. 
	
	We now sketch a proof that the reductions of Lemma $\ref{2d}$ are the only possible ones, even for higher polyspheres. The main point is that result and proof of Lemma~\ref{subspace} is generic, indeed the vector subspaces of $\mathbb{R}^2\oplus\cdots\oplus\mathbb{R}^2$ take the form
	\[
	\{ (\xi_1, \ldots, \xi_m, \ldots, A^d_1\xi_1 + \ldots + A^d_m \xi_m + t_d v_d) \;|\; \xi_1,\ldots,\xi_m \in\mathbb{R}^2, t_1,\ldots,t_d \in\mathbb{R} \}
	\]
	The argument is essentially the same as before, roughly we proceed by letting $\{u^1,\ldots,u^k\}$ be a basis for the subspace, where we can write each vector $u^i = (u^i_1,\ldots u^i_n)$, each $u^i_j\in\mathbb{R}^2$ being a $2$d column vector. Form the block matrix $U = [u]^j_i$, a rank $k$ matrix of size $2n\times k$. Form a $k\times k$ submatrix of full rank by discarding rows. Consider a vector $(x_1,\ldots,x_n)$; the form of $x_i$ depends on the discards in the rows of block $i$:
	Where no rows are discarded, we have a free $\xi_i$, if both rows are discarded we $A^i_j \xi_j$ where we discarded both rows, we have in addition the $t_j v_j$ where we discarded only one.
	
	Consider then the possible Lie triple systems $m(V)$, where $V$ takes the form above. Pick two terms from $V$; these must reduce to one of the cases described in Lemma~\ref{2d} if we set the other free terms $\xi_j,t_j$ to zero. This proves that the $A^i_j$ must be orthogonal or zero, and must be zero if paired with a $tv$ term; it remains to show that for any given $i$ only one $A^i_j$ can be non-zero. For this purpose we compute
	\begin{align*}
	[[m(Ax_1 + By_1),&m(Ax_2+By_2)],m(Ax_3 + By_3)] = \\
	\bigg(
	&A[[x_1,x_2]]A^T + A(x_1 y_2^T - x_2 y_1^T) B^T \\
	&+ B(x_2 y_1^T - x_1 y_2^T) A^T + B[[y_1,y_2]]B^T
	\bigg) (Ax_3 + By_3),
	\end{align*}
	which must equal $A[[x_1,x_2]]x_3 + B[[y_1,y_2]]y_3$ if we are to obtain a Lie triple system. Following our argument we assume that $A,B$ are both orthogonal. Now setting $x_1=-x_2=x$, and $y_1=y_2=y$ results in
	\[
	2(ACB^T - BCA^T)(Ax_3 + By_3) = 0,
	\]
	where $C=xy^T$. As $(Ax_3 + By_3)$ spans $\mathbb{R}^2$ we are left with the relation $ACB^T = BCA^T$, and as $A,B$ are orthogonal it follows that $D^T C = CD$, where $D=A^T B$. Lemma~\ref{aux mat} then shows that $D=I$ and hence $A=B$. The main equation becomes
	\[
	A\big([[x_1,x_2]]+[[y_1,y_2]]\big)(x_3+y_3) = A\big([[x_1,x_2]]x_3 + [[y_1,y_2]]y_3\big),
	\]
	and hence
	\[
	A\big([[x_1,x_2]]y_3 + [[y_1,y_2]]x_3\big) = 0,
	\]
	which cannot hold for all $x_i, y_i$ as $A$ is invertible.
\end{proof}

\end{document}